\documentclass[twoside,11pt]{article}
\usepackage{journal2e_v1}

\usepackage[numbers]{natbib}
\usepackage{epsf}
\usepackage{epsfig}
\usepackage{amsmath}
\usepackage{subfigure}
\usepackage{amssymb}
\usepackage{algorithm}
\usepackage{multirow}
\usepackage{multicol}

\usepackage{graphicx} 

\usepackage{algorithm}
\usepackage{algorithmic}

\usepackage{epsf}
\usepackage{epsfig}
\usepackage{amsmath}
\usepackage{subfigure}
\usepackage{amssymb}
\usepackage{multirow}
\usepackage{multicol}
\usepackage{pgfplots}
\usepackage{algorithm}
\usepackage{algorithmic}

\def\S{{\bf S}}

\def\S{{\bf S}}

\def\0{{\bf 0}}
\def\1{{\bf 1}}

\def\OM{{\mathcal O}}

\def\SM{{\mathcal S}}

\def\RB{{\mathbb R}}
\def\RBmn{{\RB^{m\times n}}}
\def\EB{{\mathbb E}}
\def\PB{{\mathbb P}}

\def\argmin{\mathop{\rm argmin}}

\def\nnz{\mathrm{nnz}}

\def\diag{\mathrm{diag}}

\def\RB{{\mathbb R}}

\title{Revisiting Sub-sampled Newton Methods}
\author{\name  Haishan Ye, Luo Luo   \\
	\addr {yhs12354123@gmail.com; rickyluoluo@gmail.com } \\
	Department of Computer Science and Engineering \\
	Shanghai Jiao Tong University \\
	800 Dong Chuan Road, Shanghai, China 200240	
	\AND
	\name Zhihua Zhang   \\
	\addr zhzhang@gmail.com \\
	School of Mathematical Science \\
	Peiking  University \\
	Beijing, China 100871
}
\begin{document}
	
	\maketitle
	
	\begin{abstract}
		
		Many machine learning models depend on solving a large scale  optimization problem.
		Recently, sub-sampled Newton methods have emerged to attract much attention for optimization
		due to their efficiency  at
		each iteration, rectified a weakness in the ordinary Newton method of suffering a high cost at each
		iteration while commanding a high convergence rate.
		In this work we
		propose two new efficient Newton-type methods, \emph{Refined Sub-sampled Newton} and \emph{Refined Sketch Newton}.
		Our methods  exhibit a great advantage over existing sub-sampled Newton methods, especially when Hessian-vector multiplication can be calculated efficiently and Hessian matrix is ill-conditioned.
		Specifically, the proposed methods are shown to converge superlinearly in general case and quadratically  under a little stronger assumption.
		The proposed methods can be generalized to a unifying framework for the convergence proof of several existing sub-sampled Newton methods, revealing new convergence properties.
		Finally, we empirically evaluate the performance of our methods on several standard datasets and
		the results show consistent improvement in computational efficiency.
	\end{abstract}
	
	\section{Introduction}
	We consider the following optimization problem
	\begin{equation}
	\min_{x} F(x) = \frac{1}{n}\sum_{i=1}^{n}f_i(x), \label{eq:prob_desc}
	\end{equation}
	where $x\in\RB^p$, functions $f_i:\RB^{p}\to \RB$, and $n$ is assumed to be far larger than $p$ (i.e., $n\gg p$). Many machine learning models can be expressed as \eqref{eq:prob_desc} where each $f_i$ is the loss w.r.t.\ the $i$-th sample. There are many such examples, e.g.,  logistic regressions, support vector machines, neural networks,  and graphical models.
	
	Many optimization algorithms to solve problem~\eqref{eq:prob_desc} are based on the following iteration:
	\begin{equation}
	x^{(t+1)} = x^{(t)} - \eta_t Q_t \text{g}(x^{(t)}),
	\end{equation}
	where $t$ is the number of iterations.
	If $Q_t$ is the identity matrix and $\text{g}(x^{(t)}) = \nabla F(x^{(t)})$, the resulting procedure is called \emph{Gradient Descent} (GD) which achieves sublinear convergence for general smooth convex objective and linear convergence for smooth-strongly convex ones. When $n$ is large, the full gradient method is inefficient due to its iteration cost scaling linearly in $n$. Consequently, stochastic gradient descent (SGD) has been  a typical alternation \cite{robbins1951stochastic,li2014efficient,cotter2011better}. Such a method samples a small mini-batch of data to construct an approximate gradient to achieve cheaper cost in each iteration. However, the convergence rate can be significantly slower than that of the full gradient methods \cite{nemirovski2009robust}. Thus, a  great deal of efforts have been made to devise modification to achieve the convergence rate of the full gradient while keeping low iteration cost \cite{johnson2013accelerating, roux2012stochastic,schmidt2013minimizing,Zhang}.

	If $Q_t$ is a $p\times p$ positive definite matrix of containing the curvature information, this  formulation leads us to \emph{second-order} methods. It is well known that second order methods enjoy superior convergence rate in both theory and practice compared to \emph{first-order} methods which only make use of the gradient information. The standard Newton Method, where  $Q_t = [\nabla^2 F(x^{(t)})]^{-1}$,  $\text{g}(x^{(t)}) = \nabla F(x^{(t)})$ and $\eta_t = 1$,  achieves a quadratic convergence rate for smooth-strongly convex objective functions. However, the \emph{Newton method} takes $\OM(np^2+p^3)$ cost per iteration, so it becomes extremely expensive when $n$ or $p$ are very large. As a result, one tries to construct an approximation of the Hessian in a way that the update is computationally feasible, and yet still provides sufficient second order information. One  class of such methods are quasi-Newton methods, which are a generalization of the secant method to find the root of the first derivative for multidimensional problems. The celebrated Broyden-Fletcher-Goldfarb-Shanno (BFGS) and its limited memory version (L-BFGS) are the most popular and widely used \cite{nocedal2006numerical}. They take $\OM(np+p^2)$ cost per iteration.
	
	Recently,   when $n\gg p$, a class of called \emph{sub-sampled Newton} methods have been proposed, which define an approximate Hessian matrix on a  small subset of samples.  The most naive approach is to sample a subset of functions $f_i$ randomly \cite{roosta2016sub,byrd2011use,xu2016sub} to construct a sub-sampled Hessian. \citet{erdogdu2015convergence} proposed NewSamp which solves the problem that sampled Hessian may be ill-conditioned.  When the Hessian can be written as  $B^T B$ where $B$ is an available $n\times p$ matrix, \citet{pilanci2015newton} then proposed to use sketching techniques to approximate the Hessian. Similarly, \citet{xu2016sub} proposed to sample rows of $B$ with non-uniform probability distribution. \citet{agarwal2016second} proposed an algorithm called LiSSA to approximate the inversion of Hessian directly. 
	
	In the past few years, variants of sub-sampled Newton methods have been proposed. And the convergence properties have been analyzed. However, there are several important problems related to sub-sampled Newton methods are still open.
	
	\begin{enumerate}
		\item Can sub-sampled Newton methods achieve superlinear  even quadratic convergence rate without increasing sampling number?  
		\item Is there a unifying framework to analyzing the convergence properties of sub-sampled Newton methods?
		\item Is Lipschitz continuity condition  necessary for the convergence of sub-sampled Newton methods? If not, when is needed?   
	\end{enumerate}
	
	The first problem is important both in theory and application. An optimization algorithm with superlinear and quadratic convergence is appealing in most cases. The second problem is of great significance in analyzing the convergence properties sub-sampled Newton methods. Besides, a unifying framework can provide some potential inspirations for developing more efficient sub-sampled Newton methods. The third question is also of great importance both theory and application. Without the constrain of Lipschitz continuity condition, sub-sampled Newton methods can be widely used in optimization problems. In fact, \citet{erdogdu2015convergence}  found that NewSamp can be used in training SVM which did not meet the Lipschitz continuity condition. They concluded that NewSamp can be used in optimization problems where Lipschitz continuity condition is not satisfied empirically but without any theoretical analysis.   
	
	In this paper we will answer the above problems. And we summarize our contribution as follow.
	
	\begin{enumerate}
		\item We propose the \emph{Refined Sub-sampled Newton} (ReSubNewton) and \emph{Refined Sketch Newton} (ReSkeNewton), which  converge superlinearly in general case and quadratically with a little stronger assumption without increasing sampling number which answer the first problem. To the best of our knowledge, it is the first work to show that  the sub-sampled Newton method can achieve quadratic convergence rate. To achieve superlinear convergence rate, the existing methods need to sample more and more samples as iteration goes, which may turn into the exact Newton method and lose the computational efficiency. Our methods do not need any additional samples but several matrix-vector multiplications and Hessian-vector multiplications which is cheap using a `Hessian-free' technique. Especially, when Hessian-vector multiplication can be calculated very efficient, our methods have great advantage over other existing sub-sampled Newton methods.
		\item We observe that  {sub-sampled Newton} methods can be viewed as an inexact Newton procedure~\cite{nocedal2006numerical,dembo1982inexact}. Theorem~\ref{thm:univ_frm} gives a unifying framework to analyze variants of sub-sampled Newton methods which answer the second problem. The convergence properties of sub-sampled Newton methods can be analyzed easily and systematically. In Section~\ref{sec:frame}, we analyze several important variants of sub-sampled Newton method. More importantly, Theorem~\ref{thm:univ_frm} reveals the sufficient conditions to achieve different convergence rate.
		\item Theorem~\ref{thm:univ_frm} also shows that Lipschitz continuity condition is not necessary for achieving linear and superlinear convergence. And it is needed to obtain quadratic convergence. Theorem~\ref{thm:univ_frm} not only explains  the phenomenon that NewSamp \cite{erdogdu2015convergence} can be used to train SVM  which the Lipschitz continuity condition is not satisfied, but also shows that the convergence rate is linear. Hence, we clarify the third problem and greatly widen the range of applications of variants of sub-sampled Newton methods.   
		\item We analyze the convergence properties of inexact sub-sampled Newnton method and provide a practical stop criterion to get the product of inverse of sub-sampled Hessian and gradient iteratively. We also give analysis to Newton method with sub-sampled Hessian and sub-sampled gradient and obtain new convergence properties which are preferable to previous work.
	\end{enumerate}
	\begin{algorithm}[tb]
		\caption{Sub-sample Newton.}
		\label{alg:H_subsamp}
		\begin{small}
			\begin{algorithmic}[1]
				\STATE {\bf Input:} $x^{(0)}$, $0<\delta<1$, $0<\epsilon<1$;
				\STATE Set the sample size $|\mathcal{S}| \geq \frac{16K^2\log(2p/\delta)}{\sigma^2 \epsilon^2}$.
				\FOR {$t=0,1,\dots$ until termination}
				\STATE Select a sample set $\SM$, of size $|\SM|$ and $H^{(t)} = \frac{1}{|\SM|}\sum_{j\in\mathcal{S}}\nabla^2 f_j(x^{(t)})$;
				\STATE Update $x^{(t+1)}= x^{(t)}-[H^{(t)}]^{-1}\nabla F(x^{(t)})$;
				\ENDFOR
			\end{algorithmic}
		\end{small}
	\end{algorithm}
	
	\begin{algorithm}[tb]
		\caption{Sketch Newton.}
		\label{alg:sketch_newton}
		\begin{small}
			\begin{algorithmic}[1]
				\STATE {\bf Input:} $x^{(0)}$, $0<\delta<1$, $0<\epsilon<1$;
				\FOR {$t=0,1,\dots$ until termination}
				\STATE Construct an $\epsilon\frac{\sigma}{K}$-subspace embedding matrix $S$ for $B(x^{(t)})$ and where $\nabla^2 F(x)$ of the form $ \nabla^2 F(x) = (B(x^{(t)}))^TB(x^{(t)})$, calculate $H^{(t)} = [B(x^{(t)})]^{T}S^{T}SB(x^{(t)})$;
				\STATE Update $x^{(t+1)}= x^{(t)}-[H^{(t)}]^{-1}\nabla F(x^{(t)})$;
				\ENDFOR
			\end{algorithmic}
		\end{small}
	\end{algorithm}
	\subsection{Related Work}
	\citet{byrd2011use} proposed a sub-sampled Newton method which is similar to Sub-sampled Newton (SubNewton Algorithm~\ref{alg:H_subsamp}) and approximates the product of inverse of sub-sampled Hessian and gradient by conjugate gradient. The asymptotic convergence of the method was established but without quantitative bounds in \cite{byrd2011use}. \citet{erdogdu2015convergence}  then  gave local convergence analysis of sub-sampled Newton method and proposed Newsamp (Algorithm~\ref{alg:NewSamp} in Appendix). \citet{pilanci2015newton}  first  used `sketching' within the context of Newton-like methods. The authors proposed a randomized second-order method which is based on performing an approximate Newton's step using randomly sketched Hessian and gave the detailed analysis of convergence properties. The algorithm is referred as Sketch Newton (SkeNewton Algorithm~\ref{alg:sketch_newton}) in this paper.   Similarly, \citet{xu2016sub} proposed to sketching Hessian matrix with non-uniform probability distribution. \citet{agarwal2016second} proposed an algorithm called LiSSA to approximate the inversion of Hessian directly.
	
	\citeauthor{roosta2016subI} analyzed the global and local convergence rates of variants of sub-sampled Newton methods in detail \cite{roosta2016subI,roosta2016sub}.
	The work of \cite{roosta2016subI,roosta2016sub} focused on constrained optimization problem. Our work focuses on unconstrained optimization problem and proposes a proof framework that the analysis of \cite{roosta2016sub} can also be fitted into.  Though our work focuses on unconstrained optimization problems, we can use project iteration onto the constrained set for the constrained optimization problem.
	
	\begin{algorithm}[tb]
		\caption{Refined Sub-Sample Newton.}
		\label{alg:H_subsamp_iter}
		\begin{small}
			\begin{algorithmic}[1]
				\STATE {\bf Input:} $x^{(0)}$, $0<\delta<1$, $0<\epsilon<1$, $tol$;
				\STATE Set the sample size $|\mathcal{S}| \geq \frac{16K^2\log(2p/\delta)}{\sigma^2 \epsilon^2}$.	
				\FOR {$t=0,1,\dots$ until termination}
				\STATE Select a sample set $\SM$, of size $|\SM|$ and construct $H^{(t)} = \frac{1}{|\SM|}\sum_{j\in\mathcal{S}}\nabla^2 f_j(x^{(t)})$;
				\STATE Calculate $[H^{(t)}]^{-1}$, $p^{(t)} = [H^{(t)}]^{-1}\nabla F(x^{(t)}) $ and $r^{(t)} = \nabla^2 F(x^{(t)})p^{(t)}-\nabla F(x^{(t)})$;
				\WHILE  {$\|r^{t}\| > tol$}
				\STATE Calculate $p^{(t)}_r = [H^{(t)}]^{-1} r^{(t)}$, update $p^{(t)} = p^{(t)}+p^{(t)}_r$;
				\STATE Calculate $r^{(t)} = \nabla^2 F(x^{(t)})p^{(t)}-\nabla F(x^{(t)})$;
				\ENDWHILE
				\STATE Update $x^{(t+1)}= x^{(t)}-p^{(t)} $;
				\ENDFOR
			\end{algorithmic}
		\end{small}
	\end{algorithm}
	\begin{algorithm}[tb]
		\caption{Refined Sketch Newton.}
		\label{alg:Sketch_Newton_iter}
		\begin{small}
			\begin{algorithmic}[1]
				\STATE {\bf Input:} $x^{(0)}$, $0<\delta<1$, $0<\epsilon<1$, $tol$;
				\FOR {$t=0,1,\dots$ until termination}
				\STATE Construct a $\epsilon$-subspace embedding matrix $S$ and $H^{(t)} = [B(x^{(t)})]^TS^TS B(x^{(t)})$, where $\nabla^2 F(x)$ of the form $ \nabla^2 F(x) = [B(x^{(t)})]^TB(x^{(t)})$;
				\STATE Calculate $[H^{(t)}]^{-1}$, $p^{(t)} = [H^{(t)}]^{-1}\nabla F(x^{(t)}) $ and $r^{(t)} = \nabla^2 F(x^{(t)})v^{(t)}-\nabla F(x^{(t)})$;
				\WHILE  {$\|r^{t}\| > tol$}
				\STATE Calculate $p^{(t)}_r = [H^{(t)}]^{-1} r^{(t)}$, update $p^{(t)} = p^{(t)}+p^{(t)}_r$;
				\STATE Calculate $r^{(t)} = \nabla^2 F(x^{(t)})p^{(t)}-\nabla F(x^{(t)})$;
				\ENDWHILE
				\STATE Update $x^{(t+1)}= x^{(t)}-p^{(t)} $;
				\ENDFOR
			\end{algorithmic}
		\end{small}
	\end{algorithm}
	
	\begin{algorithm}[tb]
		\caption{Preconditioned Newton-CG with Subsampled Hessian.}
		\label{alg:preconditioned Newton-CG}
		\begin{small}
			\begin{algorithmic}[1]
				\STATE {\bf Input:} $x^{(0)}$, $0<\delta<1$, $0<\epsilon<1$, $tol$;
				\FOR {$t=0,1,\dots$ until termination}
				\STATE Construct a sub-sampled Hessian or sketched Hessian $H^{(t)}$ with parameter $\epsilon$;
				\STATE Set $p_0 = 0$, $r_0 = \nabla F(x^{t})$;
				\STATE Solve $H^{(t)}y_0 = r_0$ for $y_0$ and set $d_0 = -y_0$;
				\FOR  {$i = 0,1,2,\dots$}
				\STATE Set $\alpha_i = r_i^Tr_i/d_i^T\nabla^2 F(x^{(t)})d_i$;
				\STATE Set $z_{i+1} = z_i + \alpha_id_i$;
				\STATE Set $r_{i+1} = r_i + \alpha_i\nabla^2 F(x^{(t)})d_i$
				\IF {$\|r_{i+1}\| \leq tol$}
				\STATE $p^{(t)}=-r_{i+1}$;
				\STATE Break;
				\ENDIF
				\STATE Solve $H^{(t)}y_{i+1} = r_{i+1}$;
				\STATE Set $\beta_{i+1} = r_{i+1}^{T}y_{i+1}/r_i^{T}y_i$
				\STATE Set $d_{i+1} = -y_{i+1} + \beta_{i+1}d_i$
				\ENDFOR
				\STATE Update $x^{(t+1)}= x^{(t)}-p^{(t)} $;
				\ENDFOR
			\end{algorithmic}
		\end{small}
	\end{algorithm}
	
	\section{Notation and Preliminaries}
	
	In this section, we introduce the notation and  preliminaries that will be used in this paper.
	\subsection{Notation}
	Given a matrix $A=[a_{ij}] \in \RB^{m \times n}$ of rank $ \rho $,  its SVD is given as
	$A=U\Sigma V^{T}=U_{k} \Sigma_{k} V_{k}^{T}+U_{\rho\setminus k} \Sigma_{\rho{\setminus} k} V_{\rho{\setminus}k}^{T}$,
	where $U_{k}$ and $U_{\rho{\setminus}k}$ contain the left singular vectors of $A$,  $V_{k}$ and $V_{\rho{\setminus}k}$ contain the right
	singular vectors of $A$, and $\Sigma=\diag(\sigma_1, \ldots, \sigma_{\rho})$ with $\sigma_1\geq \sigma_2 \geq \cdots \geq \sigma_{\rho}>0$ are
	the nonzero singular values of $A$. If $A$ is positive semidefinite, then $U = V$ and the square root of $A$ can be defined as $A^{1/2} = U\Sigma^{1/2}U^T$.
	
	Additionally,  $\|A\|_{F} \triangleq (\sum_{i,j}a_{ij}^{2})^{1/2}=(\sum_{i}\sigma_{i}^{2})^{1/2}$
	is the Frobenius norm of $A$ and
	$\|A\|\triangleq \sigma_{1}$ is the spectral norm. If $A$ is a positive definite matrix, $\|x\|_A\triangleq \|A^{1/2}x\|$ is called  $A$-norm. The condtion number of $A$ is defined as $\kappa(A) \triangleq \frac{\sigma_1}{\sigma_q}$. 
	
	Throughout this paper, we use notions of linear convergence rate, superlinear convergence rate and quadratic convergence rate. In our paper, the convergence rates we will use are defined in a little different way from standard ones. A sequence of vectors $\{x^{(t)}\}$ is said to converge linearly to a limit point $x^*$, if for some $0\leq \rho<1$,
	\[
	\limsup_{t\to \infty} \frac{\|\nabla^2F(x^{*})(x^{(t+1)} - x^{*})\|}{\|\nabla^2F(x^{*})(x^{(t)} - x^{*})\|} = \rho,
	\]
	where $F(x)$ the function we want to optimize.
	Similarly, superlinear convergence and quadratic convergence are respectively defined  as
	\begin{align*}
	\limsup_{t \to \infty} \frac{\|\nabla^2F(x^{*})(x^{(t+1)} - x^{*})\|}{\|\nabla^2F(x^{*})(x^{(t)} - x^{*})\|} = 0,\\
	\limsup_{t \to \infty} \frac{\|\nabla^2F(x^{*})(x^{(t+1)} - x^{*})\|}{\|\nabla^2F(x^{*})(x^{(t)} - x^{*})\|^2} = \rho.
	\end{align*}
	
	We call it as linear-quadratic convergence rate shown as below
	\[
	\|\nabla^2F(x^{*})(x^{(t+1)} - x^{*})\| \leq \rho_1\|\nabla^2F(x^{*})(x^{(t+1)} - x^{*})\| + \rho_2 \|\nabla^2F(x^{*})(x^{(t)} - x^{*})\|^2,
	\]
	where $0<\rho_1<1$.
	Besides, we assume that each $f_i$ is convex and twice differentiable. And the Lipschitz continuity condition for Hessian is defined as follows:
	\[
	\|\nabla^2F(x) - \nabla^2F(y)\| \leq L\|x-y\|,
	\]
	where $L>0$ is the Lipschitz constant.
	
	We also assume that each $f_i$ and $F$ have the following properties:
	\begin{align}
	\max_{i\leq n}\|\nabla^2f_i(x)\| \leq K < \infty,\label{eq:k} \\
	\lambda_{\min}(\nabla^2F(x^{(t)}))\geq \sigma>0. \label{eq:sigma}
	\end{align}

	\subsection{Randomized sketching matrices} \label{subsec:ske_mat}
	We first give an $\epsilon$-subspace embedding property which will be used in ReSkeNewton . Then we list some useful different types of randomized sketching matrices.
	\begin{definition} \label{lem:sub-embed}
		$S\in\RB^{s\times m}$ is an $\epsilon$-subspace embedding matrix for any fixed matrix $A\in\RB^{m\times d}$. Then, for all $x \in \RB^{d}$,
		$\|SA x\|_{2}^{2}=(1\pm\epsilon)\|Ax\|_{2}^{2}$.
	\end{definition}
	
	{\bf{Gaussian sketching matrix:}} The most classical sketching matrix is Gaussian sketching matrix $S\in\RB^{s\times m}$ with i.i.d \ normal random variables with variance $1/s$. Because of well-known concentration properties of Gaussian random matrices \cite{woodruff2014sketching}, gaussian random matrices are very attractive. Besides, $s = \OM(d/\epsilon^2)$ is enough to guarantee $\epsilon$-subspace embedding property any fixed matrix $A\in\RB^{m\times d}$. $s = \OM(d/\epsilon^2)$ is the tightest bound in known types of sketching matrices. However, Gaussian random matrices are dense matrices. It is costly to compute $SA$.
	
	{\bf Random sampling sketching matrix:}
	Let $V\in\RB^{m\times d}$  be column orthonormal basis for $A\in\RB^{m\times d}$  with $m>d$,  and $v_{i,*}$ denote the $i$-th row of $V$. Let $\ell_{i} = \|v_{i,*}\|_{F}^{2}/d$ and $s$ be an integer with $1\leq s\leq m$. Then the $\ell_{i}$'s are leverage scores for $A$. Given a distribution with $p_i \geq \beta \ell_i$ and $\sum_{1}^{m} p_i = 1$, where $0<\beta\leq 1$, we construct a sampling matrix  $\Omega\in\RB^{m\times s}$ and a rescaling matrix $D\in\RB^{s \times s}$ as follows. For every $j = 1,\dots,s$, independently and with replacement, pick an index $i$ from the set $\{1,2\dots,m\}$ with probability  $p_{i}$ and set $\Omega_{ij} = 1$ and $D_{jj}=1/\sqrt{p_{i}s}$. The random sampling sketching matrix $S$ for $A$ is then defined as $S = \Omega D$. To achieve an $\epsilon$-subspace embedding property for $A$, $s = \OM(d\log d/(\beta\epsilon^2))$ is needed. When we sample by the distribution with leverage sores, i.e. $\beta=1$, we just need $s = \OM(d\log d/\epsilon^2)$ samples. There are several methods to approximate the leverage score of $A$ which are of computational efficiency \cite{drineas2012fast,li2013iterative}. 
	
	{\bf Sparse embedding matrix:} Sparse embedding matrix $S\in\RB^{s\times m}$ is of the form that there is only one non-zero entry uniformly sampling from $\{1,-1\}$ in each column \cite{clarkson2013low}. Hence the it is very efficient to compute $SA$, especially when $A$ is a sparse matrix. To achieve an $\epsilon$-subspace embedding property for $A\in\RB^{m\times d}$, $s = \OM(d^2/\epsilon^2)$ is sufficient \cite{meng2013low,woodruff2014sketching}. 
	
	For random sampling, $\mu$-Coherence is an important concept which is closely relatd to leverage scores.
	\begin{definition}
		Let $V\in\RB^{m\times d}$  be column orthonormal basis for $A\in\RB^{m\times d}$  with $m>d$,  and $v_{i,*}$ denote the $i$-th row of $V$. Then the $\mu$-Coherence of $A$ is
		\[
		\mu(A) = \frac{n}{d} \max_i\|v_{i,*}\|^2_2
		\]
	\end{definition}
	
	Other types of sketching matrices like Subsampled Randomized Hadamard Transformation and detailed properties of sketching matrices and subspace embedding matrices can be found in the survey \cite{woodruff2014sketching}.
	
	\subsection{Computation cost of matrix operations}
	We will give the computation cost of basic matrix operations, the cost can be found in Matrix Computation \cite{golub2012matrix}. 
	
	For matrix multiplication, given dense matrices $B\in\RBmn$ and $C\in\RB^{n\times k}$, the basic cost of the matrix product $B\times C$ is $\OM(mnk)$ flops. It costs $\OM(k\cdot\nnz(B))$ flops for the matrix product $B\times C$ when $B$ is sparse, where $\nnz(B)$ denotes the number of nonzero entries of $B$. If $S$ is a sparse subspace embedding matrix \cite{woodruff2014sketching}, then the product $S\times B$ costs $\OM(\nnz(B))$. 
	
	For SVD and QR-decomposition of a  matrix $A\in\RB^{m\times n}$, it costs about $\OM(mn^2)$ flops if $n\leq m$ \cite{golub2012matrix}. To get the inverse of a positive-definite matrix $A\in\RB^{n\times n}$, it costs  $\OM(n^3)$ flops by Cholesky decomposition.
	
	\section{Inexact Newton}
	
	The basic Newton step is to calculate a direction vector $v^{(t)}_N$ by solving the following symmetric $p\times p$ linear system
	\begin{equation}
	\nabla^2 F(x^{(t)}) v^{(t)}_N = \nabla F(x^{(t)}). \label{eq:newtion_equa}
	\end{equation}
	An inexact Newton method tries to find an approximation $v^{(t)}$ to $v^{(t)}_N$. We define the residual with $v^{(t)}$ as follows
	\begin{equation}
	r^{(t)} = \nabla^2 F(x^{(t)})v^{(t)}-\nabla F(x^{(t)}). \label{eq:res_def}
	\end{equation}
	Usually, inexact Newton methods should satisfy the following condition
	\begin{equation}
	\|r^{(t)}\|\leq \gamma^{(t)}\|\nabla F(x^{(t)})\|, \label{eq: res_cond}
	\end{equation}
	where the sequence $\{\gamma^{(t)}\}$ (with $0<\gamma^{(t)}<1$ for all $t$) is called the forcing sequence.
	
	We give a new form of convergence properties of inexact Newton method. The convergence properties can be found in \cite{nocedal2006numerical,dembo1982inexact}.
	\begin{theorem} \label{thm:inexact_newton}
		Suppose that $\nabla^2 F(x)$ exists and is continuous in a neighborhood of a minimizer $x^*$, with $\nabla^2 F(x^*)$ being positive definite. Consider the iteration $x^{(t+1)} = x^{(t)} - v^{(t)}$ where $v^{(t)}$ satisfies \eqref{eq: res_cond}. If the starting point $x^{(0)}$ is sufficiently near $x^{*}$, then the sequence $\{x^{(t)}\}$ converges to $x^*$ and satisfies
		\begin{equation}
		\|\nabla^2F(x^{*})(x^{(t+1)} - x^{*})\| \leq (\gamma^{(t)} + 6M\eta)\|\nabla^2F(x^{*})(x^{(t)} - x^*) \|. \label{eq:conv_lin}
		\end{equation}
		where $M \equiv\max(\|\nabla^{2}F(x^*)\|,\|\nabla^{2}F(x^*)^{-1}\|)$ and $\eta = o(1)$ are constant. 
		
		Besides, if  $\nabla^{2}F(x)$ is Lipschitz continuous for $x$ near $x^{*}$, then sequence $\{x^{(t)}\}$ converges to $x^*$ and satisfies
		\begin{equation}
		\|\nabla^2F(x^{*})(x^{(t+1)} - x^{*})\|  \leq \gamma^{(t)}\|\nabla^2F(x^{*})(x^{(t)} - x^*) \|+ 14LM^2\|\nabla^2F(x^{*})(x^{(t)} - x^*) \|^2, \label{eq:Lip_conv}
		\end{equation}	
		where $L$ is the Lipschitz constant.
	\end{theorem}
	
	\section{Refined Sub-sampled Newton Methods}
	We propose ReSubNewton and ReSkeNewton which can both achieve superlinear and quadratic convergence rate without more and more sampling as iteration goes. The key advantage of our algorithms is to get accurate approximation to $[\nabla^{2}F(x^{(t)})]^{-1}\nabla F(x^{(t)})$ without any additional samples.
	
	\subsection{Algorithms Description}
	In ReSubNewton, we first select a sample set $\SM$ to construct $H^{(t)}$ such that
	\begin{equation}
	(1-\epsilon)H^{(t)}  \preceq \nabla^{2}F(x^{(t)}) \preceq (1+\epsilon) H^{(t)}. \label{eq:H_cond}
	\end{equation}
	Then, $p^{(t)} = [H^{(t)}]^{-1}\nabla F(x^{(t)}) $ is a good approximation to $[\nabla^{2}F(x^{(t)})]^{-1}\nabla F(x^{(t)})$. Furthermore, if $\|r^{(t)}\|$ is bigger than a prespecified tolerance $tol$, where $r^{(t)} = \nabla^2 F(x^{(t)})p^{(t)}-\nabla F(x^{(t)})$, we can refine $p^{(t)}$ as follows
	\begin{equation*}
	p^{(t)} = p^{(t)} + [H^{(t)}]^{-1}r^{(t)}.
	\end{equation*}
	After $k$ refine iterations above, $\|r^{(t)}\|$ will decrease to $\epsilon^{k}\|r^{(t)}_0\|$, where $r^{(t)}_0$ is the residual before refine iterations. If we choose $\epsilon$ properly, $\|r^{(t)}\|$ will decrease to $tol$ very fast. Finally, we update $x^{(t+1)}$ with refined $p^{(t)}$ just as
	\[
	x^{(t+1)} = x^{(t)} - p^{(t)}.
	\]
	The detailed algorithm is depicted in Algorithm~\ref{alg:H_subsamp_iter}.
	
	In machine learning, it is common that the Hessian matrix is of the form $B(x^{(t)})^TB(x^{(t)})$ and $B(x^{(t)})$ is an explicitly available $n\times p$, e.g., SVM, generalized linear models, etc.. Hence, we propose Refined Sketch Newton  (Algorithm~\ref{alg:Sketch_Newton_iter}) where refine iterations are used similar to ReSubNewton.  Different types of sketching matrices have the same algorithms structure of ReSkeNewton. But, the computational cost of ReSkeNewton will be different if different types of sketching matrices used in algorithm. The most popular two kinds of sketching matrices are leverage-score sketching matrix and sparse embedding matrix because they can achieve sketching in input sparsity.

	We give the properties of ReSubNewton and ReSkeNewton in the following theorems.
	\begin{theorem} \label{thm:sub_refine}
		Let Assumptions.~\eqref{eq:k} and Eqn.~\eqref{eq:sigma} hold, and  $0<\delta<1$	and $0<\epsilon<1/2$ be given. If the sample size is set as $|\mathcal{S}| \geq \frac{16K^2\log(2p/\delta)}{\sigma^2 \epsilon^2}$, then  Algorithm~\ref{alg:H_subsamp_iter} has the following convergence properties:
		\begin{enumerate}
			\item If $tol/\|\nabla F(x^{(t)})\|\to 0$, then the sequence $\{x^{(t)}\colon  t = 1,\dots,T\}$ converges superlinearly.
			\item If $tol = \OM(\|\nabla F(x^{(t)})\|^2)$ and $\nabla^{2}F(x^{(t)})$ is Lipschitz continuous, then the sequence $\{x^{(t)}\colon  t = 1,\dots,T\}$ converges quadratically.
		\end{enumerate}
		Besides, iterations of the inner loop of Algorithm~\ref{alg:H_subsamp_iter} are at most $\frac{\log \frac{K\|\nabla F(x^{(t)})\|}{\sigma tol}}{\log \frac{1}{\epsilon}}$.
	\end{theorem}
	
	\begin{theorem}\label{thm:sketch_refine}
		Let Assumptions~\eqref{eq:k} and Eqn.~\eqref{eq:sigma} hold, and  $0<\delta<1$	and $0<\epsilon<1/2$ be given. If $S$ is an $\epsilon$-subspace embedding matrix for $B(x^{(t)})$, then Algorithm~\ref{alg:Sketch_Newton_iter} has the following convergence properties:
		\begin{enumerate}
			\item If $tol/\|\nabla F(x^{(t)})\|\to 0$, then the sequence $\{x^{(t)}\colon t = 1,\dots,T\}$ converges superlinearly.
			\item If $tol = \OM(\|\nabla F(x^{(t)})\|^2)$ and $\nabla^{2}F(x^{(t)})$ is Lipschitz continuous, then the sequence $\{x^{(t)}\colon t = 1, \dots, T\}$ converges quadratically.
		\end{enumerate}
		Besides, iterations of the inner loop are at most $\frac{\log \frac{K\|\nabla F(x^{(t)})\|}{\sigma tol}}{\log \frac{1}{\epsilon}}$.
	\end{theorem}
	
	When $\nabla^2F(x^{(t)}) =\frac{1}{n}\sum_{i=1}^{n}\nabla^{2}f_i(x^{(t)}) =\frac{1}{n}\sum_{i=1}^{n}B(i,:)^TB(i,:) = B^TB$, where $\nabla^{2}f_i(x^{(t)}) = B(i,:)^TB(i,:)$ and $B\in\RB^{n\times p}$, ReSubNewton can be viewed as a special case of ReSkeNewton. And random sampling sketching matrix can be constructed just as described in Subsection~\ref{subsec:ske_mat}. 
	\begin{corollary} \label{cor:rand_samp}
		Let Assumptions.~\eqref{eq:k} and Eqn.~\eqref{eq:sigma} hold, and  $0<\delta<1$	and $0<\epsilon<1/2$ be given. And Hessian matrix $\nabla^2F(x^{(t)})$ is of the form $\frac{1}{n}B^T B$ and $\nabla^{2}f_i(x^{(t)}) = B(i,:)^TB(i,:)$, where $B \in \RB^{n\times p}$. If the sample size is set as $|\mathcal{S}| \geq \OM(\frac{\mu(B(x^{(t)}))p\log(p/\delta)}{\epsilon^2})$, then  Algorithm~\ref{alg:H_subsamp_iter} has the same convergence properties described in Theorem~\ref{thm:sub_refine}.
	\end{corollary}
	
	In fact, our algorithms can be recast as preconditioned Newton-CG using sub-sampled Hessian as preconditioner just as described in Algorithm~\ref{alg:preconditioned Newton-CG}. In Algorithm~\ref{alg:preconditioned Newton-CG}, we use $H^{(t)}$ as preconditioner which satisfies Equation~\eqref{eq:H_cond}. The convergence properties of Algorithm~\ref{alg:preconditioned Newton-CG} is the same to Algorithm~\ref{alg:H_subsamp_iter} and~\ref{alg:Sketch_Newton_iter}.
	\subsection{Algorithm Analysis}
	We conduct comparison between ReSubNewton and ReSkeNewton. First, they have different application scenarios. ReSubNewton is suitable for problems of the form~\eqref{eq:prob_desc} and ReSkeNewton applies to the problem where Hessian is of the form $[B(x^{(t)})]^T B(x^{(t)})$ and $B(x^{(t)})$ ($n\times p$) is explicitly available. Furthermore, ReSkeNewton has a better theoretical property; that is, to achieve \eqref{eq:H_cond}, the sampled size $|\SM|$ of ReSubNewton depends on $K^2/\sigma^{2}$ linearly, where $K/\sigma$ is commonly referred to as the condition number, and sketched dimension $\ell$ of ReSkeNewton is independent on $K^2/\sigma^{2}$, i.e., condition number independent. Hence, the sketched dimension can be small even Hessian matrix is ill-conditioned which is very attractive in practice.
	
	In our algorithms, $p^{(t)}$ is refined to approximate $[\nabla^{2}F(x^{(t)})]^{-1}\nabla F(x^{(t)})$ more accurately within the inner loop. The main calculation cost of the inner loop is matrix multiplications and $\nabla^2 F(x^{(t)})p^{(t)}$. Note that $\nabla^2 F(x^{(t)})p^{(t)}$ can be calculated cheaply without explicit Hessian by a `Hessian-free' technique \cite{byrd2011use,nocedal2006numerical}. Especially, in many machine learning problems the loss functions $f_i$ take the following linearly-parameterized form: $f_i(x) = \ell(b_i,a_i^{T}x)$, where $b_i$ is the label of the $i$-th input data $a_i$, thus $\nabla^{2}F(x^{(t)})p^{(t)}$ can be computed in input sparsity of the data matrix which is very efficient when the data matrix is sparse. Besides, iterations of the inner loop are at most $\log \frac{K\|\nabla F(x^{(t)})\|}{\sigma tol}/\log \frac{1}{\epsilon}$ by Theorems~\ref{thm:sub_refine} and \ref{thm:sketch_refine}, hence, the number of iterations of the inner loop is small if we choose a moderate small $\epsilon \in (0,1)$. 
	
	What's more, our algorithms are very robust because $\|r^{(t)}\| \leq tol$ is satisfied in each iteration. We can extend our method to other sub-sampled Newton methods, the failure probability of algorithms can be reduced to $0$ by checking the $\|r^{(t)}\|$ defined in \eqref{eq:res_def} with little additional cost.
	
	\subsection{Application Range}
	ReSubNewton and ReSkeNewton both need to compute the inverse of sub-sampled or sketched Hessian which cost $\OM(p^3)$ arithmetic operations by Cholesky decomposition in general case which is similar to SubNewton and SkeNewton. Hence, ReSubNewton and ReSkeNewton are more suitable for the cases where $p$ is moderate or small. 
	
	In fact, our algorithms still have advantages in real applications even when $p$ is large. Without loss of generality, we conduct comparison between ReSkeNewton with other existing algorithms and we assume that $\nabla^{2}F(x^{(t)}) = [B(x^{(t)})]^T B(x^{(t)})$. Similar result also holds for ReSubNewton.  When $p$ is large, the \citet{byrd2011use} proposed to use sub-sampled Newton method with conjugate gradient to approximate $[H^{(t)}]^{-1}\nabla F(x^{(t)})$. The convergence rate of conjugate gradient linearly depends on $\sqrt{\kappa(H^{(t)})}$ i.e. $\kappa(\tilde{B}^{(t)})$ because $H^{(t)} = [\tilde{B}(x^{(t)})]^T \tilde{B}(x^{(t)})$. It costs $\OM(|\mathcal{S}|p\kappa(\tilde{B}^{(t)}))$ to approximate $[H^{(t)}]^{-1}\nabla F(x^{(t)})$, when $\tilde{B}^{(t)}$ is a dense matrix. As to our algorithm, it requires $\OM(|\mathcal{S}|p^2)$ to compute $H^{(t)}$ and $\OM(np+p^3)$ to compute the inner loop. Therefore, our algorithm has comparable or better performance when $\kappa(\tilde{B}^{(t)})$ is large. The running time of LiSSA~\cite{agarwal2016second} also depends on condition number. Hence, even when $p$ is large, our algorithms are competitive because $\tilde{B}^{(t)}$ is commonly ill-conditioned in machine learning application. 
	
	When $B^{(t)}$ is sparse which is also common in machine learning problems, we sketch $B(x^{(t)})$ to get $\tilde{B}(x^{(t)})$ using leverage-score sampling. Then we conduct QR decomposition to get $\tilde{B}(x^{(t)}) = UR$ with Givens operations in a proper order. This decomposition is fast and $R$ is a sparse matrix because $\tilde{B}(x^{(t)})$ is a sparse matrix. For similar reason, $R^{-1}$ is sparse and can be computed efficiently. Besides, in the inner loop $H(x^{(t)})^{-1}r^{(t)}$ can be computed efficiently in the manner $H(x^{(t)})^{-1}r^{(t)} = ({R^{-T}}(R^{-1}r^{(t)}))$. Combining Hessian-free technique, the inner loop is very efficient when $B(x^{(t)})$ is sparse. 
	
	Hence, our algorithms have good performance in real machine learning problem. 
	
	\subsection{Comparison with previous work}
	Next, we compare our main algorithms with other main variants of sub-sampled Newton methods.
	
	First, our algorithms are both condition number independent when Hessian matrix is of the form $B^{T}B$. The previous variants of sub-sampled Newton methods are all condition number dependent. Hence, our algorithms need much less samples than previous samples which means faster speed.
	
	Before our work, several sub-sampled Newton methods with superlinear convergence rate have been proposed. \citet{roosta2016sub}  and \citet{pilanci2015newton} proposed to reduce the value of $\epsilon$ in NaSubNewton and SkeNewton to $0$ as iteration goes. Though these methods can achieve superlinear convergence rate, the sub-sampled size $|\SM^{(t)}|$ or the sketched dimension $\ell$ will go beyond $n$ which will become the exact Newton method and lose computational efficiency. Furthermore, for $\epsilon = \frac{1}{\log(1+t)}$ suggested in \cite{pilanci2015newton}, it will not accelerate convergence much in real applications though it converges superlinearly. Thus, our algorithms are the first practical sub-sample Newton method with superlinear convergence rate. 
	
	Finally, our algorithms can achieve quadratic convergence rate when  $\nabla^{2}F(x^{(t)})$ is Lipschitz continuous. It is the first time that variant of sub-sampled Newton achieve quadratic convergence rate.
	
	We summarize our comparisons in Table~\ref{tb:cmp}.
	\begin{table}[]
		\centering
		\caption{Comparisons between variants of sub-sampled Newton($\kappa$ is conditon number)}
		\label{tb:cmp}
		\begin{tabular}{|c|c|c|c|}
			\hline
			& $\kappa$-independent & \begin{tabular}[c]{@{}c@{}}superlinear without \\ additional sampling\end{tabular} & \begin{tabular}[c]{@{}c@{}}achieve quadratic \\ convergence\end{tabular} \\ \hline
			SubNewton\cite{byrd2011use,roosta2016sub}   & No                   & No                                                                                & No                                                                       \\ \hline
			SkeNewton\cite{pilanci2015newton}   & No                   & No                                                                                & No                                                                       \\ \hline
			Newsamp\cite{erdogdu2015convergence}     & Ye                   & No                                                                                & No                                                                       \\ \hline
			LiSSA\cite{agarwal2016second}       & No                   & No                                                                                & No                                                                       \\ \hline
			ReSubNewton & Yes, when $H=B^{T}B$   & Yes                                                                               & Yes                                                                      \\ \hline
			ReSkeNewton & Yes                  & Yes                                                                               & Yes                                                                      \\ \hline
		\end{tabular}
	\end{table}
	\section{Beyond Refined Sub-sampled Newton Methods} \label{sec:frame}
	In this section, we bring in the perspective of inexact Newton to analyze variants of sub-sampled Newton Methods and propose a unifying framework.
	
	\subsection{Unifying Framework}
	
	The existing variants of sub-sampled Newton methods have close relationship. For example, if we let $tol$ be big enough, ReSubNewton and ReSkeNewton will reduce to SubNewton and SkeNewton, respectively. In fact, NewSamp \cite{erdogdu2015convergence}, LiSSA \cite{agarwal2016second}, sub-sampled Newton with conjugate gradient \cite{byrd2011use} and sub-sampled Newton with non-uniformly sampling \cite{xu2016sub}, they all can be cast into inexact Newton.  We give Theorem~\ref{thm:univ_frm} which is a framework of analyzing convergence properties of variants of sub-sampled Newton methods and the basis of designing new sub-sampled Newton type algorithms.
	
	\begin{theorem}\label{thm:univ_frm}
		Suppose that $\nabla^2 F(x)$ exists and is continuous in a neighborhood of a minimizer $x^*$, with $\nabla^2 F(x^*)$ being positive definite. Assuming that $H^{(t)}$ is the sub-sampled Hessian and $v^{(t)} = [H^{(t)}]^{-1}\nabla F(x^{(t)})$. Then $r^{(t)}$ defined in \eqref{eq:res_def} satisfies the following property:
		\begin{align*}
		&\|r^{(t)}\| \leq \gamma^{(t)}\|\nabla F(x^{(t)})\|, \\
		& \mbox{where } \;  \gamma^{(t)} \triangleq \|(\nabla^2F(x^{(t)}) - H^{(t)}) [H^{(t)}]^{-1}\|.
		\end{align*}
		For sub-sampled newton update $x^{(t+1)} = x^{(t)} - v^{(t)}$, we have the following convergence properties:
		\begin{enumerate}
			\item If $0<\gamma^{(t)} < 1$ for all $t$, then the sequence $\{x^{(t)}\}$ converges to optimal $x^{*}$ linearly with rate $\gamma^{(t)}<\gamma<1$.
			\item If $0<\gamma^{(t)} < 1$ for all $t$ and $\gamma^{(t)} \to 0$, then the sequence $\{x^{(t)}\}$ converges to optimal $x^{*}$ superlinearly.
			\item If $0<\gamma^{(t)} < 1$ for all $t$, and $\nabla^{2}F(x)$ is Lipschitz continuous for $x$ near $x^{*}$, the sequence $\{x^{(t)}\}$ has linear-quadratic convergence rate
			\[
			\|\nabla^2F(x^{*})(x^{(t+1)} - x^{*})\|  \leq \gamma^{(t)}\|\nabla^2F(x^{*})(x^{(t)} - x^*) \|+ 14LM^2\|\nabla^2F(x^{*})(x^{(t)} - x^*) \|^2,
			\] 
			where $M$ is a constant and $L$ is the Lipschitz constant. Hence, sequence $x^{(t)}$ starts with a quadratic rate of convergence which will transform to linear rate later to optimal $x^{*}$.
			\item If $0<\gamma^{(t)} < 1$ for all $t$,  $\nabla^{2}F(x)$ is Lipschitz continuous for $x$ near $x^{*}$ and  $\gamma^{(t)} = \OM(\|\nabla F(x^{t})\|)$, then the sequence $\{x^{(t)}\}$ converges to optimal $x^{*}$ quadratically.
		\end{enumerate}
	\end{theorem}
	
	From Theorem~\ref{thm:univ_frm}, we can find something important insights. 
	
	First, Lipschitz continuity of $\nabla^{2}F(x)$ is not necessary for linear convergence and superlinear convergence rate of sub-sampled Newton methods. This reveals the reason for NewSamp can be used in training SVM where Lipschitz continuity is not satisfied. 
	Lipschitz continuity condition is only needed to get a quadratic convergence or linear-quadratic convergence. This explains the phenomena that LiSSA\cite{agarwal2016second}, NewSamp~\cite{erdogdu2015convergence}, sub-sampled Newton with non-uniformly sampling \cite{xu2016sub}, Sketched Newton \cite{pilanci2015newton} etc. all have linear-quadratic convergence rate because they all assume that Hessian is Lipschitz continuous. 
	
	Second, Theorem~\ref{thm:univ_frm} provide sufficient conditions to get superlinear and quadratic convergence rate, the sufficient condition relies on $\gamma^{(t)}$. Hence, any method which  decreases $\gamma^{(t)}$ can achieve the better convergence property. 
	
	Note that the convergence properties of ReSubNewton and ReSkeNewton can be derived from Theorem~\ref{thm:univ_frm} easily, and the detailed proof is in Appendix~\ref{sec:proof_resubnewton}. In the following subsections, we analyze several important variants of sub-sampled Newton methods.
	
	\subsection{Regularized Subsampled Newton}
	When the Hessian matrix $\nabla^2F(x^{(t)})$ is ill-conditioned, we need to a lot of samples to guarantee sub-sampled Newton methods work since the sample size is dependent on condition number. Similarly, sketched Newton methods have to sketch to a higher dimension. To reduce the influence of condition number, regularized sub-sampled Newton methods are proposed. The main two algorithms are depicted in Algorithm~\ref{alg:NewSamp} and~Algorithm~\ref{alg:reg_subsamp}.  
	\begin{theorem} \label{thm:NewSamp}
		Assumption~\eqref{eq:k} holds, and $\lambda_{i}^{(t)}$ is the $i$-th eigenvalue of $H_{|\SM^{(t)}|}$ defined in Algorithm~\ref{alg:NewSamp}. Then $\gamma^{(t)}$ defined in Theorem~\ref{thm:univ_frm} has the following bound:
		\[
		\gamma^{(t)} \leq 1-\eta^{(t)}\frac{\lambda_p^{(t)}}{\lambda^{(t)}_{r+1}}+ \eta^{(t)}\frac{4K}{\lambda_{r+1}^{(t)}}\sqrt{\frac{\log (2p/\delta)}{|\SM^{(t)}|}}  \triangleq \xi^{(t)}.
		\]
		If $\xi^{(t)} < 1$  (by choosing $\eta^{(t)}$ and $|\SM^{(t)}|$ properly), then  Algorithm~\ref{alg:NewSamp} has the following convergence properties:
		\begin{enumerate}
			\item The sequence $\{x^{(t)} \colon  t = 1, \dots, T\}$ converges linearly with probability $(1-\delta)^{T}$.
			\item If $\nabla^{2}F(x^{(t)})$ is Lipschitz continuous, then the  sequence $\{x^{(t)}\colon t = 1, \dots, T\}$ has a linear-quadratic convergence rate  with probability $(1-\delta)^{T}$.
		\end{enumerate}
	\end{theorem}
	Theorem~\ref{thm:NewSamp} explains the empirical results that Newsamp is applicable in training SVM which the Lipschitz continuity condition is not satisfied \cite{erdogdu2015convergence}. It is worth pointing out that the result about  the linear convergence rate without Lipschitz continuity is unknown before.
	
	\begin{theorem} \label{thm:Reg_subnewton}
		Assumption~\eqref{eq:k} holds, and $\lambda_{i}^{(t)}$ is the $i$-th eigenvalue of $H_{|\SM^{(t)}|}=\frac{1}{|\SM|}\sum_{j\in\mathcal{S}}\nabla^2 f_j(x^{(t)})$. Then $\gamma^{(t)}$ defined in Theorem~\ref{thm:univ_frm} has the following bound:
		\[
		\gamma^{(t)} \leq 1-\frac{\lambda_p^{(t)}}{\lambda^{(t)}_{p}+\alpha}+ \frac{4K}{\lambda_{p}^{(t)}+\alpha}\sqrt{\frac{\log (2p/\delta)}{|\SM^{(t)}|}}  \triangleq \xi^{(t)}.
		\]
		If $\xi^{(t)} < 1$  (by choosing $|\SM^{(t)}|$ properly), then  Algorithm~\ref{alg:reg_subsamp} has the following convergence properties:
		\begin{enumerate}
			\item The sequence $\{x^{(t)} \colon  t = 1, \dots, T\}$ converges linearly with probability $(1-\delta)^{T}$.
			\item If $\nabla^{2}F(x^{(t)})$ is Lipschitz continuous, then the  sequence $\{x^{(t)}\colon t = 1, \dots, T\}$ has a linear-quadratic convergence rate  with probability $(1-\delta)^{T}$.
		\end{enumerate}
	\end{theorem}	
	
	We conduct comparison between Theorem~\ref{thm:NewSamp} and Theorem~\ref{thm:Reg_subnewton}. If we set $\eta^{(t)}=1$ in Theorem~\ref{thm:NewSamp} and set $\alpha$ to satisfy $\lambda_p^{(t)}+\alpha = \lambda_{r+1}^{(t)}$ in Theorem~\ref{thm:Reg_subnewton}, then we can see that the convergence properties of Theorem~\ref{thm:NewSamp} and Thereom~\ref{thm:Reg_subnewton} are the same. Algorithm~\ref{alg:reg_subsamp} do not need to perform SVD comparing to Algorithm~\ref{alg:NewSamp}. In fact, Algorithm~\ref{alg:NewSamp} proposes a method to choose $\alpha$ in Algorithm~\ref{alg:reg_subsamp}.  
	\subsection{Inexact Subsampled Newton}
	In light of inexact Newton method, $[H^{(t)}]^{-1}\nabla F(x^{(t)})$ can also be computed inexactly by optimizing the following problem
	\begin{align}
	\argmin_{p} \frac{1}{2} p^T H^{(t)} p - p^T\nabla F(x^{(t)}) \label{eq: inexact_p}
	\end{align}
	This scheme has been used in several work \cite{byrd2011use,roosta2016subI,xu2016sub}. 
	Conjugate gradient is the most popuplar method to solve above problem \cite{byrd2011use,xu2016sub}. When optimization problem \eqref{eq: inexact_p} solved inexactly, we have the following property.
	
	\begin{theorem} \label{thm:inexact_p}
		We assume that $H^{(t)}$ is the sub-sampled Hessian and $p^{(t)}$ is an approximate solution satisfying $\|H^{(t)}p^{(t)} - \nabla F(x^{(t)})\| \leq \epsilon_0\|\nabla F(x^{(t)})\|$, where $0<\epsilon_0<1$. We define $\epsilon_1 \triangleq \|(\nabla^2F(x^{(t)}) - H^{(t)}) [H^{(t)}]^{-1}\|$ and $\gamma^{(t)} \triangleq \epsilon_0 + (1+\epsilon_0) \epsilon_1$. 
		If $\gamma^{(t)} < 1$, for update $x^{(t+1)} = x^{(t)} - p^{(t)}$, we have the following convergence properties: 
		\begin{enumerate}
			\item The sequence $\{x^{(t)} \colon  t = 1, \dots, T\}$ converges linearly.
			\item If $\nabla^{2}F(x^{(t)})$ is Lipschitz continuous, then the  sequence $\{x^{(t)}\colon t = 1, \dots, T\}$ has a linear-quadratic convergence rate.
		\end{enumerate}
	\end{theorem} 
	
	Lemma 7 in \cite{xu2016sub} gives a similar convergence result to Theorem~\ref{thm:inexact_p}. However, Theorem~\ref{thm:inexact_p} are preferable due to two advantages. The first one is that the condition $\|H^{(t)}p^{(t)} - \nabla F(x^{(t)})\| \leq \epsilon_0\|\nabla F(x^{(t)})\|$ can be easily checked in optimization procedure. The condition is $\|p^{(t)} - [H^{(t)}]^{-1}\nabla F(x^{(t)})\| \leq \epsilon_0 \|[H^{(t)}]^{-1}\nabla F(x^{(t)})\|$ in \cite{xu2016sub}. This condition is very hard to check in the procedure of solving \eqref{eq: inexact_p}, hence, it has to stop the optimization iteration by experience. The second one is that Theorem~\ref{thm:inexact_p} does not need Lipschitz continuity condition to get a linear convergence rate.
	\begin{algorithm}[tb]
		\caption{Sub-sample Hessian and Gradient.}
		\label{alg:H_G_subsamp}
		\begin{small}
			\begin{algorithmic}[1]
				\STATE {\bf Input:} $x^{(0)}$, $0<\delta<1$, $0<\epsilon<1$;
				\STATE Set the sample size $|\mathcal{S}_H|$ and $|\mathcal{S}_{\text{g}}|$.
				\FOR {$t=0,1,\dots$ until termination}
				\STATE Select a sample set $\SM_H$, of size $|\SM|$ and construct $H^{(t)} = \frac{1}{|\SM|}\sum_{j\in\mathcal{S}}\nabla^2 f_j(x^{(t)})$;
				\STATE Select a sample set $\SM_{\text{g}}$ of size $|\SM_{\text{g}}|$ and calculate $\text{g}(x^{(t)}) = \frac{1}{|\SM_{\text{g}}|}\sum_{i\in\SM_{\text{g}}}\nabla f_i(x^{(t)})$.
				\STATE Update $x^{(t+1)}= x^{(t)}-[H^{(t)}]^{-1}\text{g}(x^{(t)})$;
				\ENDFOR
			\end{algorithmic}
		\end{small}
	\end{algorithm}
	\subsection{Subsampled Hessian and Gradient}
	In fact, we can also subsample gradient to accelerate sub-sampled Newton method depicted in Algorithm~\ref{alg:H_G_subsamp} \cite{byrd2011use,roosta2016sub}. 
	
	\begin{theorem} \label{thm:sub_gradient}
		We assume that $H^{(t)}$ is the sub-sampled Hessian and $g(x^{(t)})$ is a sub-sampled gradient constructed in Algorithm~\ref{alg:H_G_subsamp} satisfying $\|\text{g}(x^{t}) - \nabla F(x^{(t)})\| \leq \epsilon_0 \|\nabla F(x^{(t)})\|$. We define $\epsilon_1 \triangleq \|(\nabla^2F(x^{(t)}) - H^{(t)}) [H^{(t)}]^{-1}\|$ and $\gamma^{(t)} \triangleq \epsilon_0 \|\nabla^2F(x^{(t)})[H^{(t)}]^{-1}\| + \epsilon_1$. 
		If $\gamma^{(t)} < 1$, for update $x^{(t+1)} = x^{(t)} - p^{(t)}$, where $p^{(t)} = [H^{(t)}]^{-1}\text{g}(x^{(t)})$, Algorithm~\ref{alg:H_G_subsamp} have the following convergence properties: 
		\begin{enumerate}
			\item The sequence $\{x^{(t)} \colon  t = 1, \dots, T\}$ converges linearly.
			\item If $\nabla^{2}F(x^{(t)})$ is Lipschitz continuous, then the  sequence $\{x^{(t)}\colon t = 1, \dots, T\}$ has a linear-quadratic convergence rate.
		\end{enumerate}
	\end{theorem}
	Theorem 13 in \cite{roosta2016sub} got an R-linear convergence rate which is $\|x^{(t)}-x^*\| \leq \rho^{t}\sigma$, where $0<\rho<1$ and $\sigma$ is a constant. To get R-linear convergence with rate $\rho$, it needs $\|\text{g}(x^{t}) - \nabla F(x^{(t)})\| \leq \epsilon_0 \|\nabla F(x^{(t)})\|$ and $\epsilon_0 = \rho^{t} \epsilon$, where $0<\epsilon<1$ is a constant. This means that it has to increase sampling gradients quickly as iteration goes. As a contrast, Theorem~\ref{thm:sub_gradient} need to increase the number of sampled gradients much slower to achieve linear since convergence since $\epsilon_0$ is a constant. Hence, our result is more attractive. 
	
	In common case, sub-sampled gradient $\text{g}(x^{(t)})$ needs to subsample over $80$ percents samples to guarantee $\|\text{g}(x^{t}) - \nabla F(x^{(t)})\| \leq \epsilon_0 \|\nabla F(x^{(t)})\|$ as iteration goes. \citet{roosta2016sub} showed that it needs $|\SM_{\text{g}}|\geq\frac{G(x^{(t)})^2}{\epsilon_0^2}$, where $G(x^{(t)}) = \max_i\|\nabla f_i(x^{(t)})\|$ for $i=1,\dots,n$. When $x^{(t)}$ is close to $x^{*}$, $G(x^{(t)})$  is large in common cases. This is the reason why Newton method and variants of sub-sampled Newton methods are very sensitive to the accuracy of sub-sampled gradient.
	
	\subsection{Discussion}
	In fact, the perspective of inexact Newton procedure may provide potential inspirations for developing more efficient sub-sampled Newton methods. For example, \citet{byrd2011use} proposed use to conjugate method to solve $[H^{(t)}]^{-1}\text{g}(x^{(t)})$ approximately, where $H^{(t)}$ and $\text{g}(x^{(t)})$ are sub-sampled Hessian and sub-sampled gradient respectively. This is a method combining inexact Subsampled Newton with sub-sampled gradient.  
	\section{Empirical Study}
	
	In this section we present experimental evaluation for our algorithms. We perform the experiments for binary classification problems. We use the following popular and standard Ridge Logistic Regression
	\begin{equation*}
	F(x) = \frac{1}{n}\sum_{i=1}^{n} \log (1+\exp(-b_i\langle a_i, x\rangle)) + \frac{\lambda}{2}\|x\|^2,
	\end{equation*}
	where $a_i \in \RB^{p}$ is the $i$-th input vector, $b_i$ is the $i$-th label and $n$ is the number of training samples. A small value of $\lambda$ often means a hard optimization problem of Ridge Logistic Regression for sub-sampled Newton methods. We perform optimization over Ridge Logistic Regression on four data sets: \emph{Nomao}, \emph{Covertype}, \emph{a9a}, and \emph{w8a}. We give the detailed description of the datasets in Table~\ref{tb:data}.
	
	\begin{table}[]
		\centering
		\caption{Datasets summary(sparsity$=\frac{\#\text{Non-Zero Entries}}{n\times p}$)}
		\label{tb:data}
		\begin{tabular}{|c|c|c|c|l|}
			\hline
			Dataset & $n$        & $p$      & sparsity       & source \\ \hline
			Nomao   &  $34465$   & $119$    &    dense   &   \cite{candillier2012design}   \\ \hline
			Covertype   &  $581012$   & $54$    &    dense   &   \cite{Blackard1999Comparative} \\ \hline
			a9a     &  $32561$   & $123$    &    $11.28\%$    &   \cite{platt199912}     \\ \hline
			w8a     &  $49749$   & $300$    &    $3.88\%$     &     \cite{platt199912}     \\ \hline
		\end{tabular}
	\end{table}
	
	We set $tol = \min(0.1,\sqrt{\|\nabla F(x^{(t)})\|})\|\nabla F(x^{(t)})\|$ in ReSubNewton and Newton-CG. Subsampled Newton with conjuate gradient(SNCG) will obtain a $p^{(t)}$ satisfying $\|H^{(t)}p^{(t)} - \nabla F(x^{(t)})\| \leq 0.05\cdot\|\nabla F(x^{(t)})\|$. In our experiment, we implement ReSubNewton as Algorithm~\ref{alg:preconditioned Newton-CG}. Besides, the first sevral iterations of ReSubNewton are implemented in Subsampled Newton with conjugate gradient to achieve faster speed. This scheme is reasonable for ReSubNewton because it will reduce to SubNewton if we set $tol$ big enough in Algorithm~\ref{alg:H_subsamp_iter}.  
	
	First, we compare ReSubNewton with Subsampled Newton with conjugate gradient. In the experiment, ReSubNewton and SNCG will subsample the same number of samples. We will change the sampling number and compare their convergence properties. We conduct our experiment on 'w8a' and set $\lambda = 0.0001$. The result is shown in Figure~\ref{fig:conv_comp}. As we can see, ReSubNewton is very robust. It converges superlinearly even when there are only $2.5$ percents of samples. In contrast, SNCG converges linearly only when sampling $20$ percents. This is because ReSubNewton is independent of condition number just as Corollary~\ref{cor:rand_samp} shows. However, the sampling number of SubSampled Newton and Sketched Newton both depends on condition number \cite{roosta2016sub,pilanci2015newton}. That is the reason why ReSubNewton is much robust than Subsampled Newton. Similarly, ReSkeNewton is also independent of condition number. 
	\begin{figure}[!ht]
		\subfigtopskip = 0pt
		\begin{center}
			\centering
			\subfigure[\textsf{$20$ percents of samples .}]{\includegraphics[width=45mm]{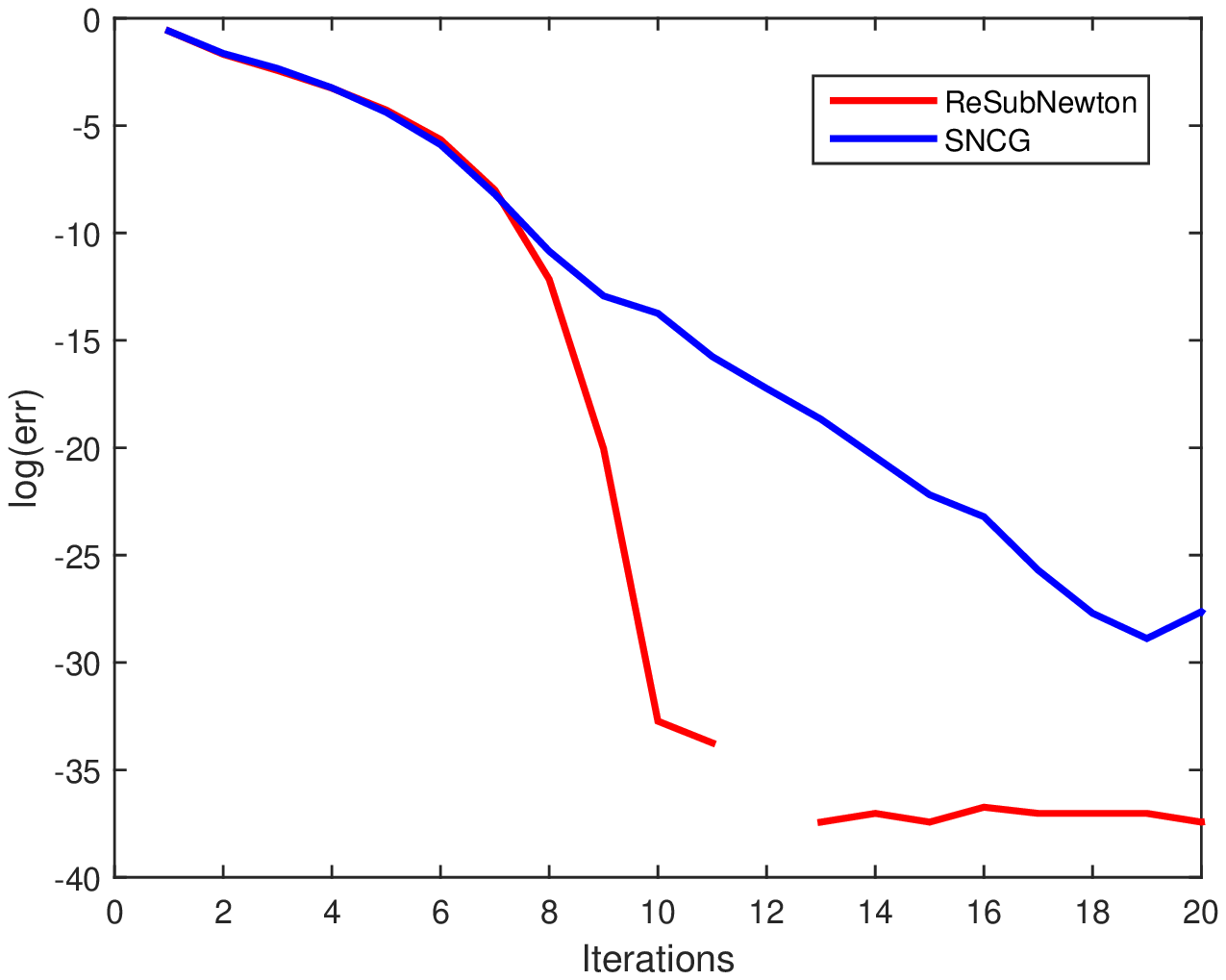}}~
			\subfigure[\textsf{$5$ percents of samples .}]{\includegraphics[width=45mm]{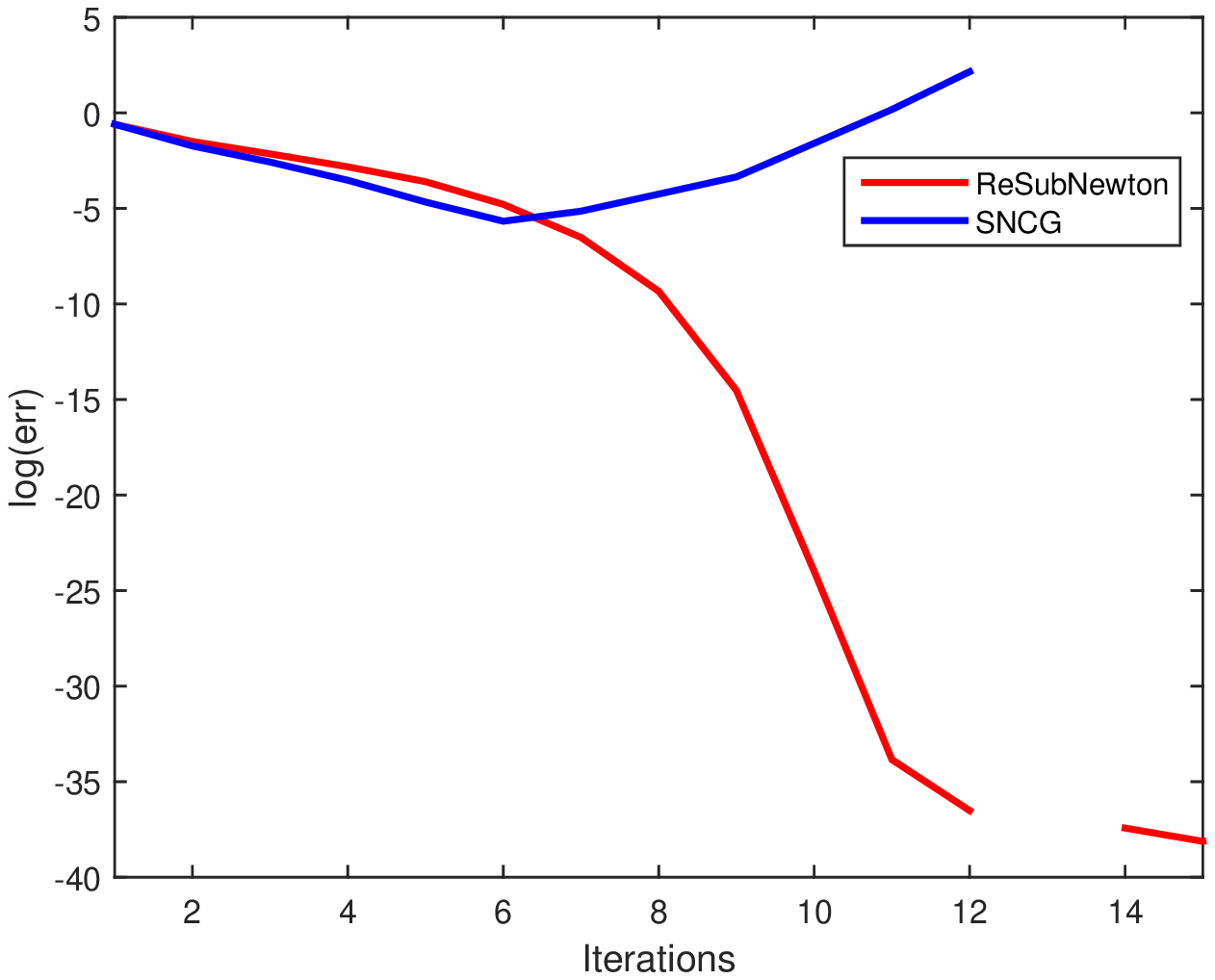}}~
			\subfigure[\textsf{$2.5$ percents of samples .}]{\includegraphics[width=45mm]{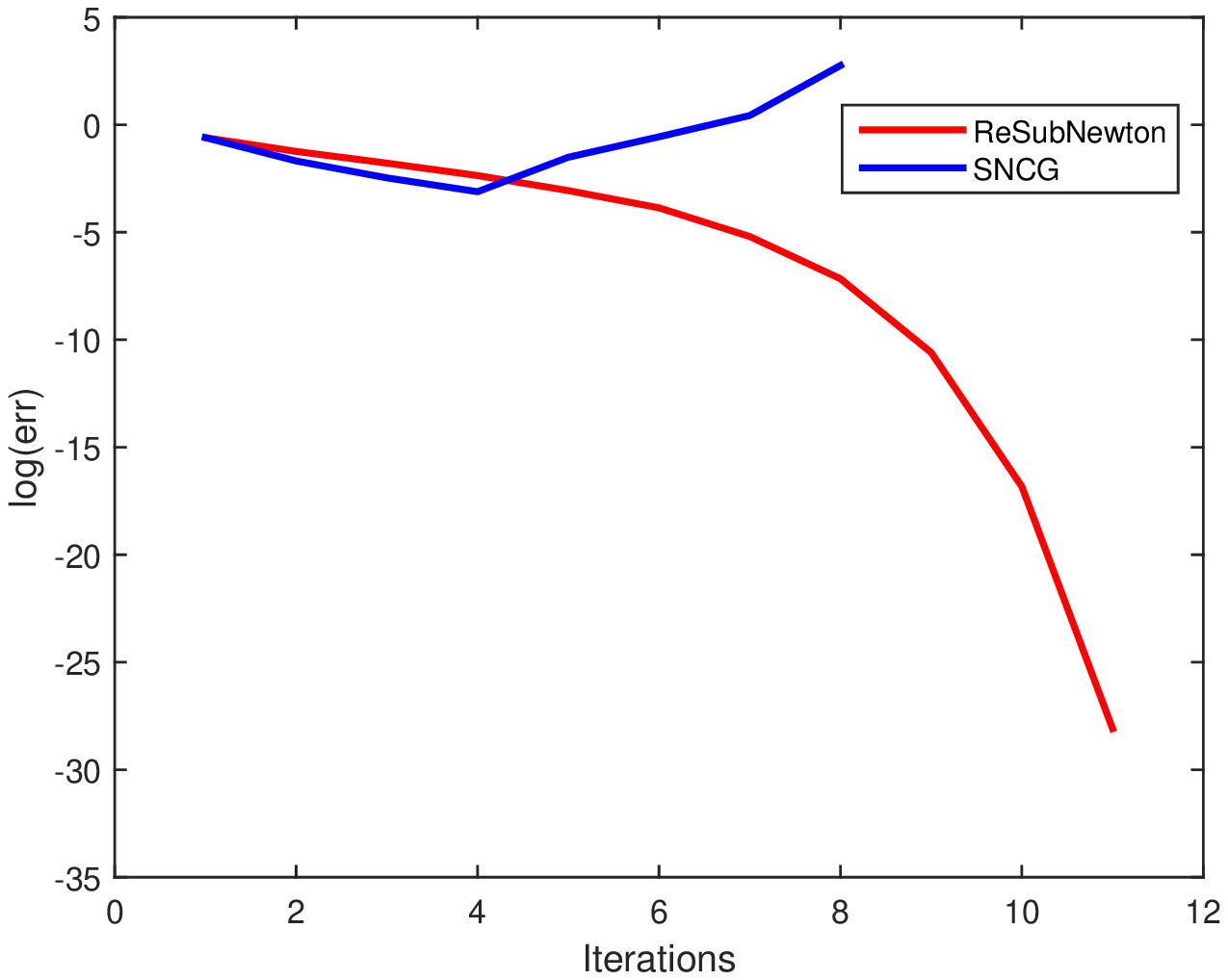}}\\
		\end{center}
		\caption{Convergence properties on different sample size.}
		\label{fig:conv_comp}
	\end{figure}
	
	Then, we compare ReSubNewton with BFGS and several important algorithms, including the standard Newton method, Newton-CG(NCG), sub-sampled Newton with conjuate gradient(SNCG). In this paper, we do not compare with SkeNewton and ReSkeNewton because the Hessian of Ridge Logistic Regression if of the form $\nabla^2F(x^{(t)}) =\frac{1}{n}\sum_{i=1}^{n}\nabla^{2}f_i(x^{(t)}) + \lambda I=\frac{1}{n}\sum_{i=1}^{n}B(i,:)^TB(i,:) +\lambda I= B^TB + \lambda I$. And Corollary~\ref{cor:rand_samp} implies that SubNewton and ReSubNewton is a special case of SkeNewton and ReSkeNewton respectively in this case with some tranformation. In the experiment, we set different value to $\lambda$ to compare the performance of algorithms in different condition number.
	
	We report results in Figure~\ref{fig:Nomao},~\ref{fig:ctp},~\ref{fig:a9a} and~\ref{fig:w8a}. As we can see, ReSubNewton achieves superlinear convergence rates on all datasets which are even close to quadratic convergence rates on all datasets. SNCG starts with quadratic convergence rate and transform into linear convergence rate. Furthermore, we can find that  ReSubNewton is very robust to the value of $\lambda$ and datasets from experiments. When $\lambda = 10^{-5}$, other algorithms all perform poorly, ReSubNewton still keeps a superlinear convergence rate and fast speed. Besides, SNCG, NCG and BFGS all show poor performance on 'covertype', however, ReSubNewton convergence very fast and has great advantages on running time. In our experiments, SNCG, NCG and BFGS all show sensitivity to the value of $\lambda$ and dataset. They have good performance, when $\lambda$ is big which means a well-conditioned Hessian matrix. However, when $\lambda$ is small which leads to ill-conditioned Hessian matrix in our experiments, they obtain poor performance. Hence, ReSubNewton is more robust and efficient than SNCG, NCG and BFGS.
	
	\begin{figure}[!ht]
		\subfigtopskip = 0pt
		\begin{center}
			\centering
			\subfigure{\includegraphics[width=45mm]{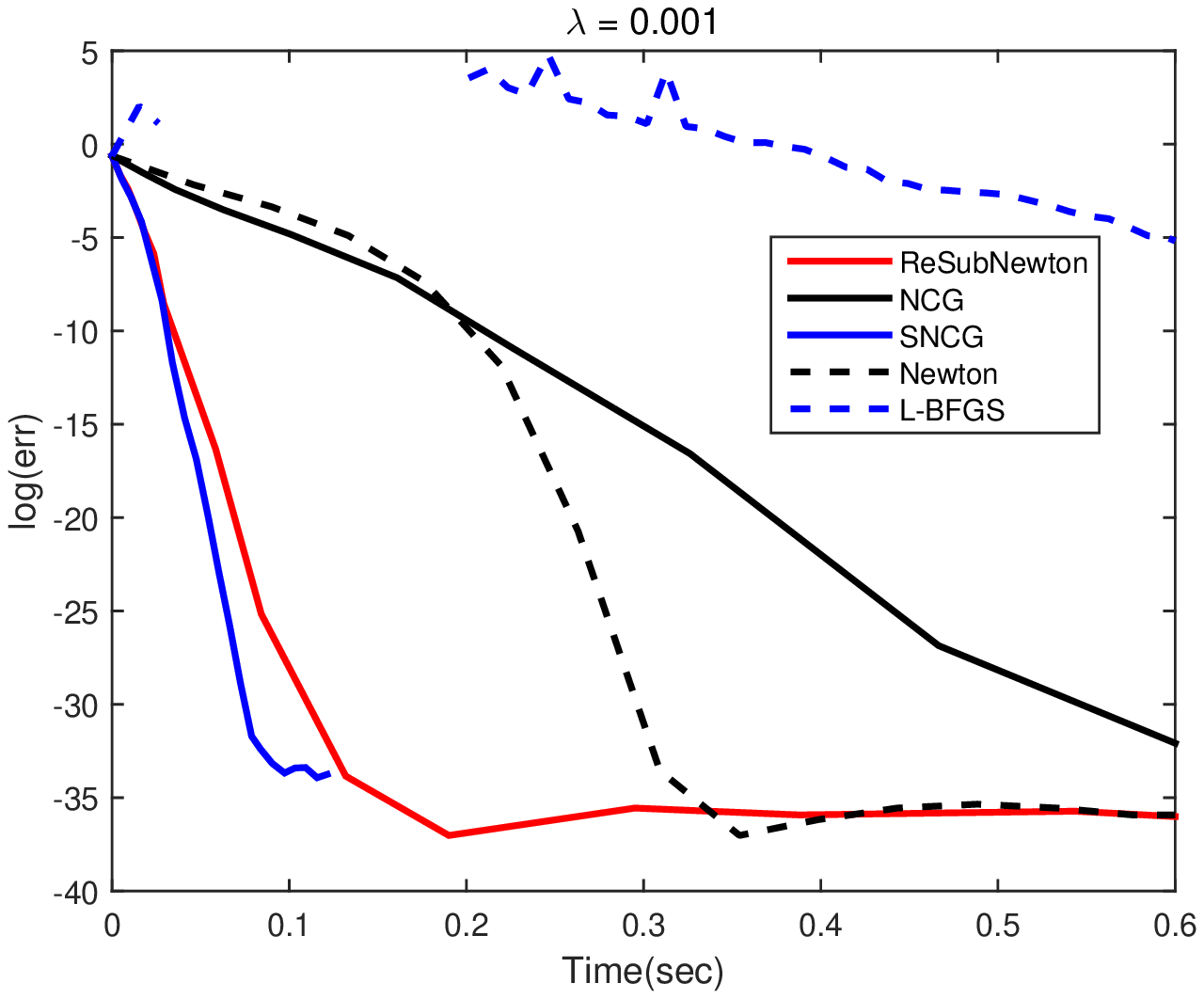}}~
			\subfigure{\includegraphics[width=45mm]{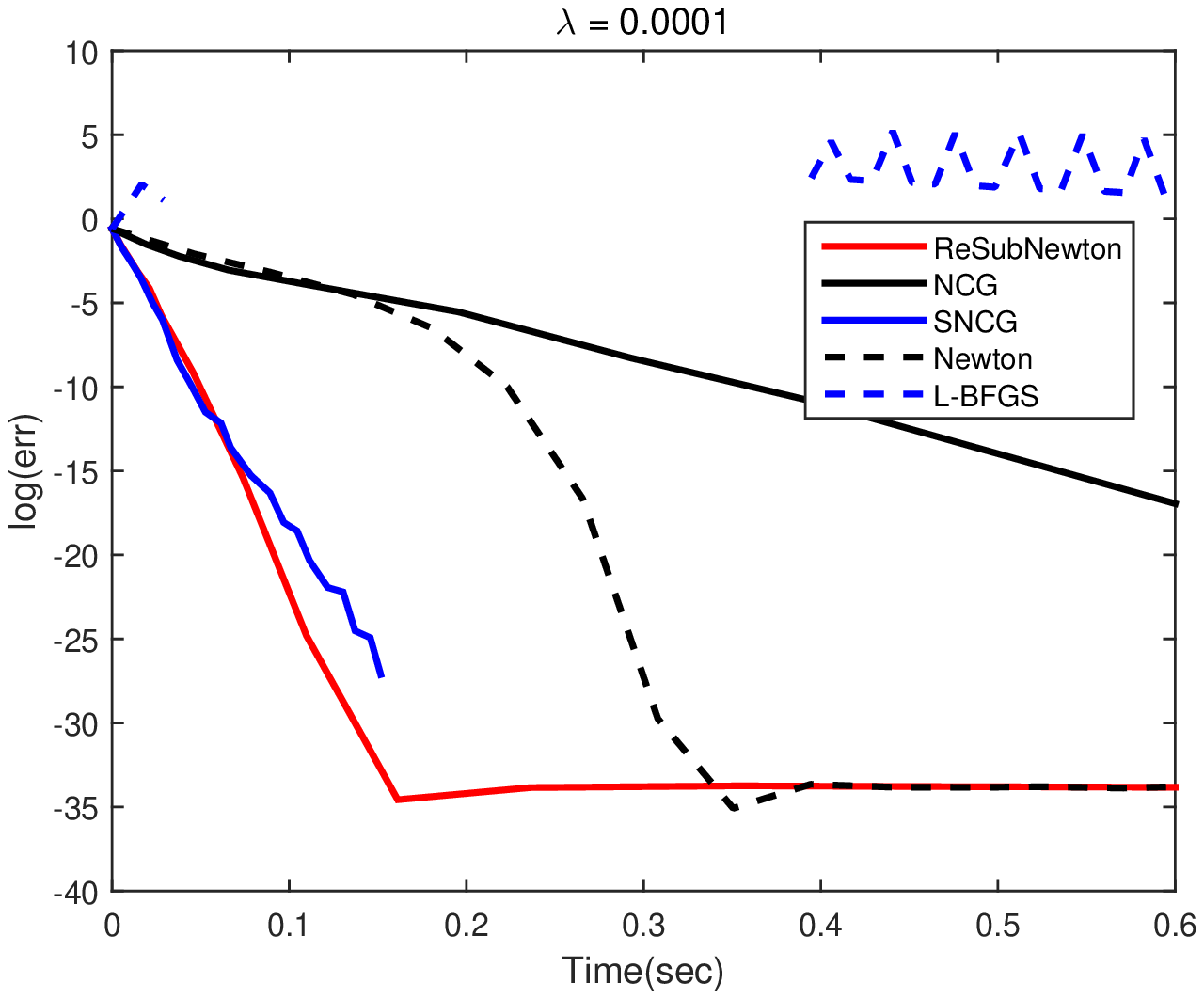}}~
			\subfigure{\includegraphics[width=45mm]{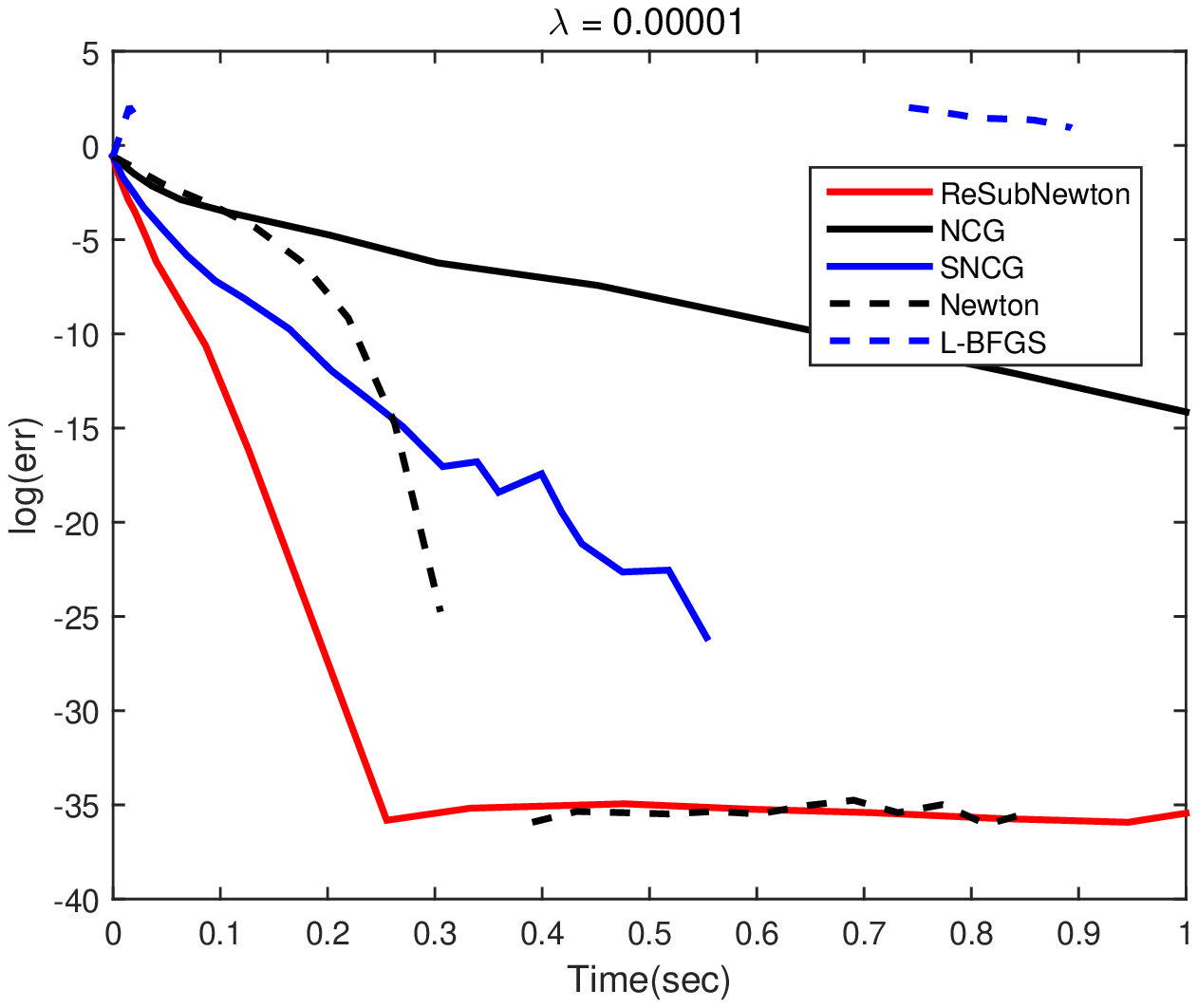}}\\
			\subfigure{\includegraphics[width=45mm]{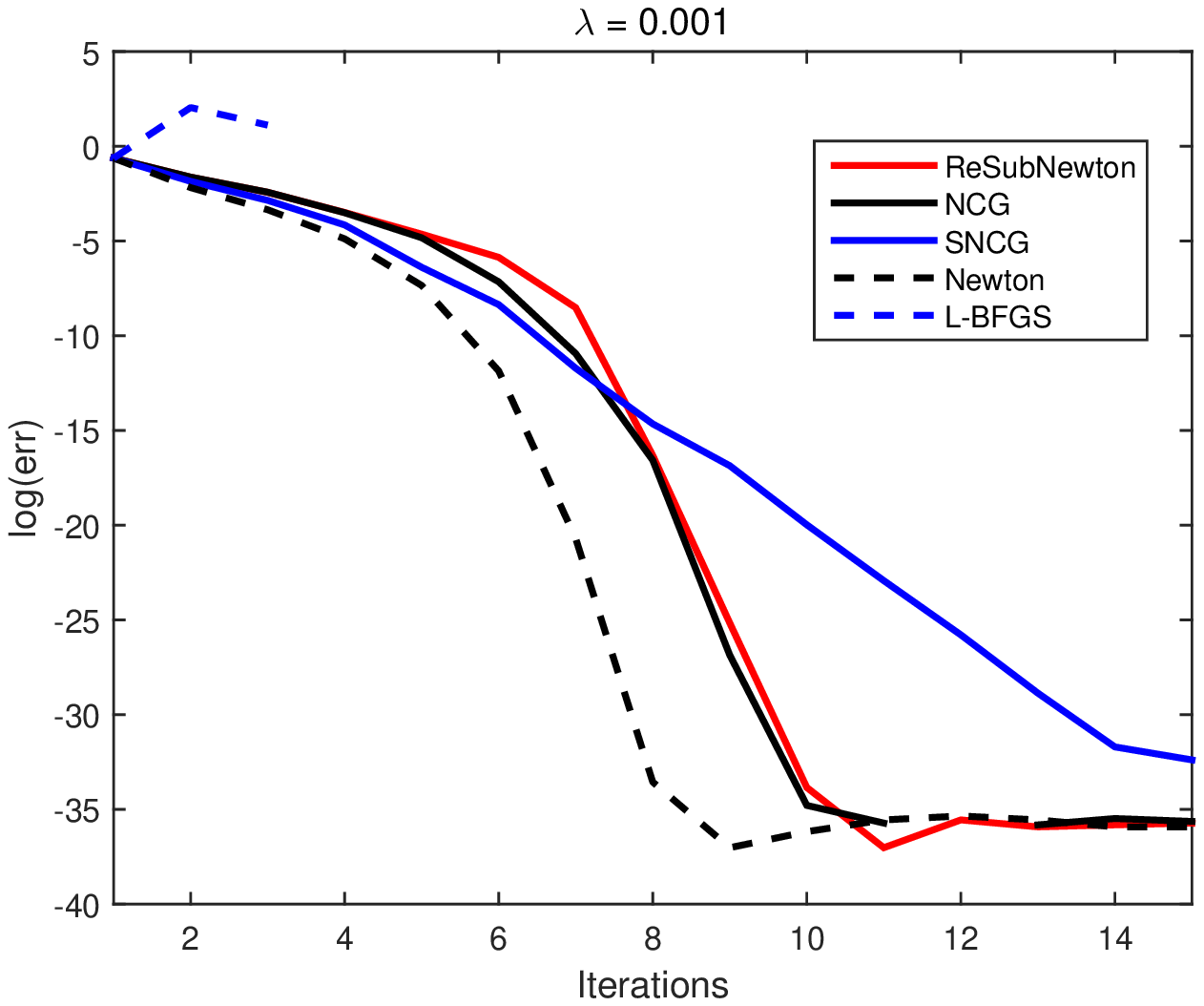}}~
			\subfigure{\includegraphics[width=45mm]{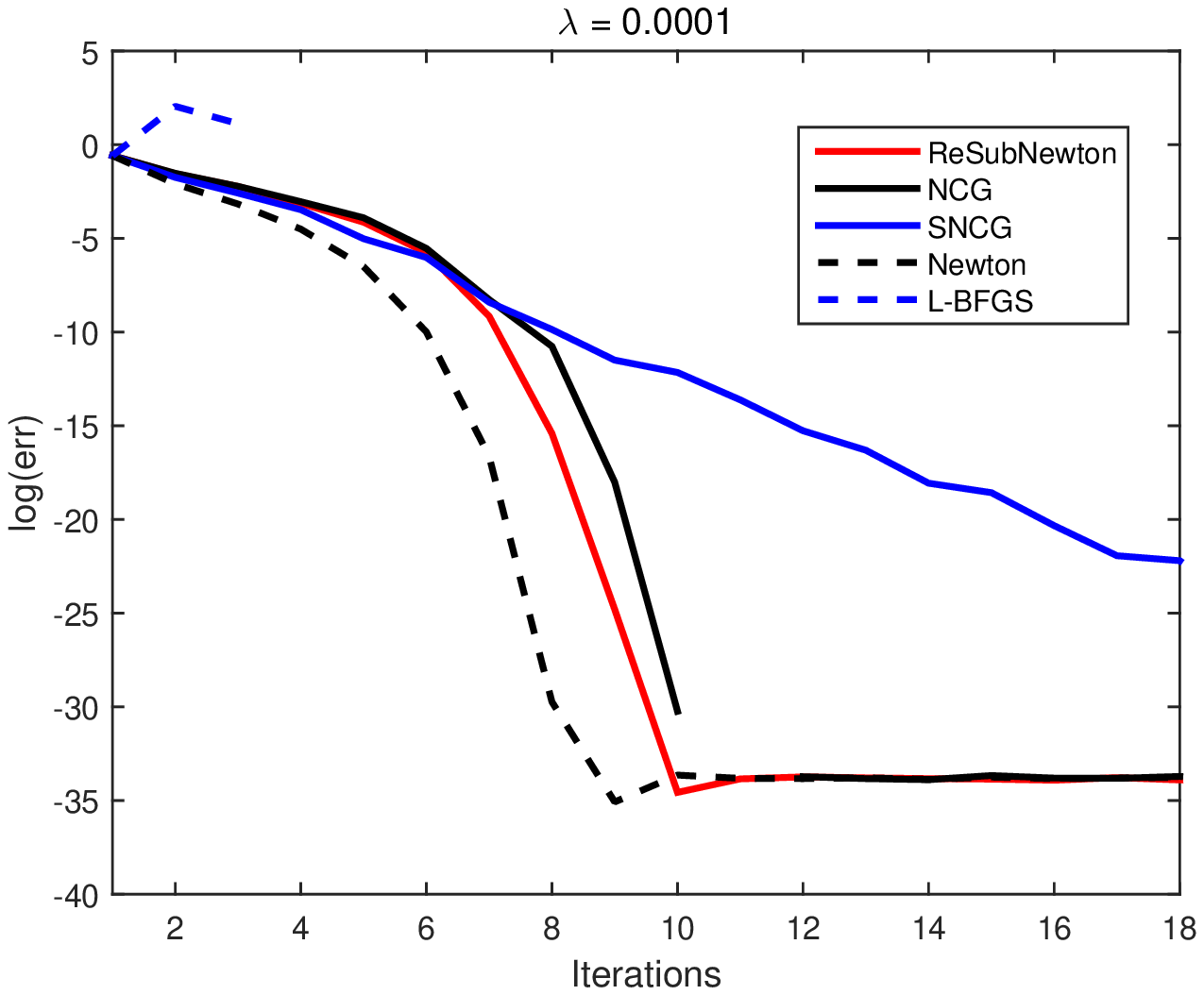}}~
			\subfigure{\includegraphics[width=45mm]{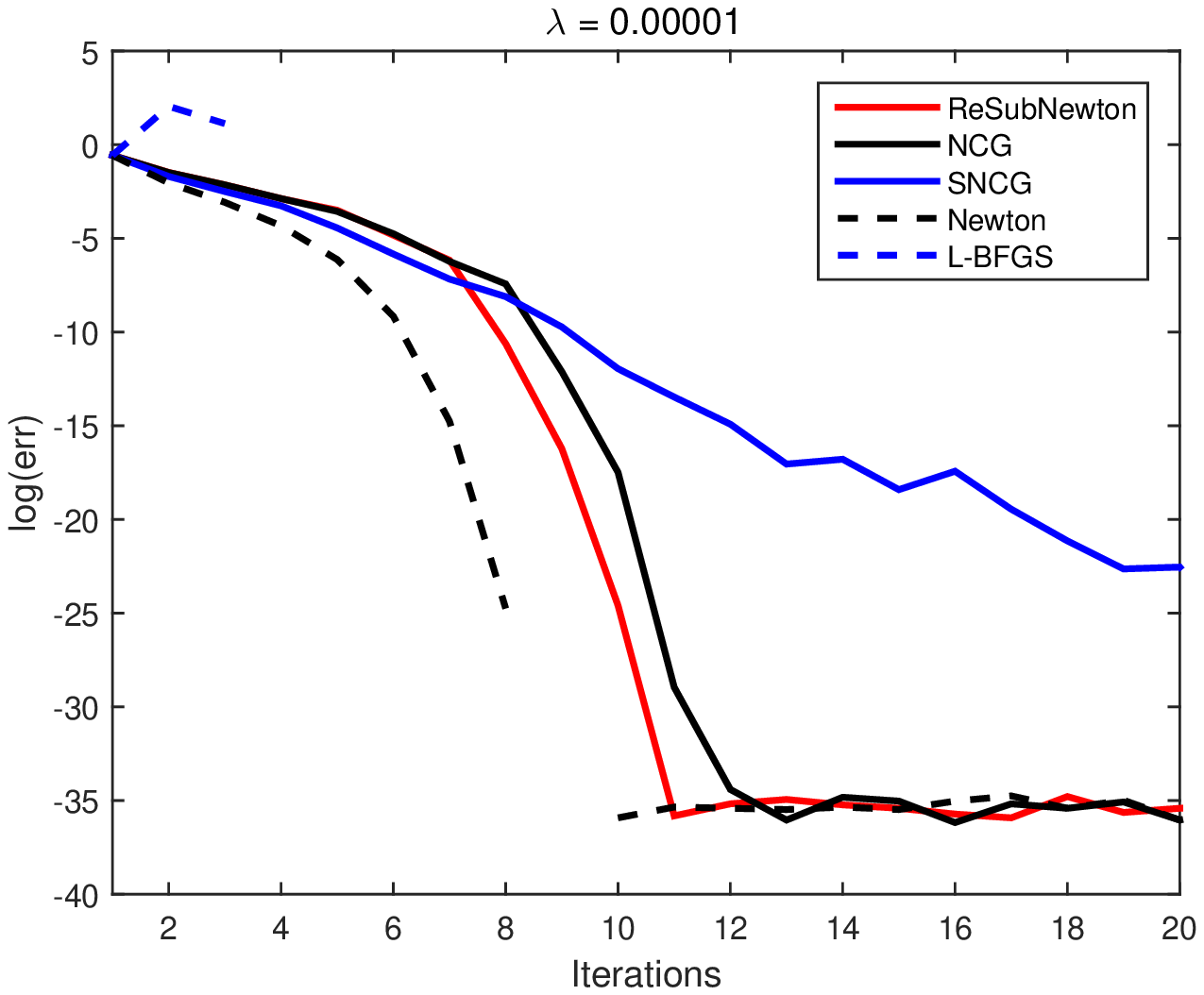}}
		\end{center}
		\caption{Experiment on 'Nomao' with different $\lambda$}
		\label{fig:Nomao}
	\end{figure}
	
	\begin{figure}[!ht]
		\subfigtopskip = 0pt
		\begin{center}
			\centering
			\subfigure{\includegraphics[width=45mm]{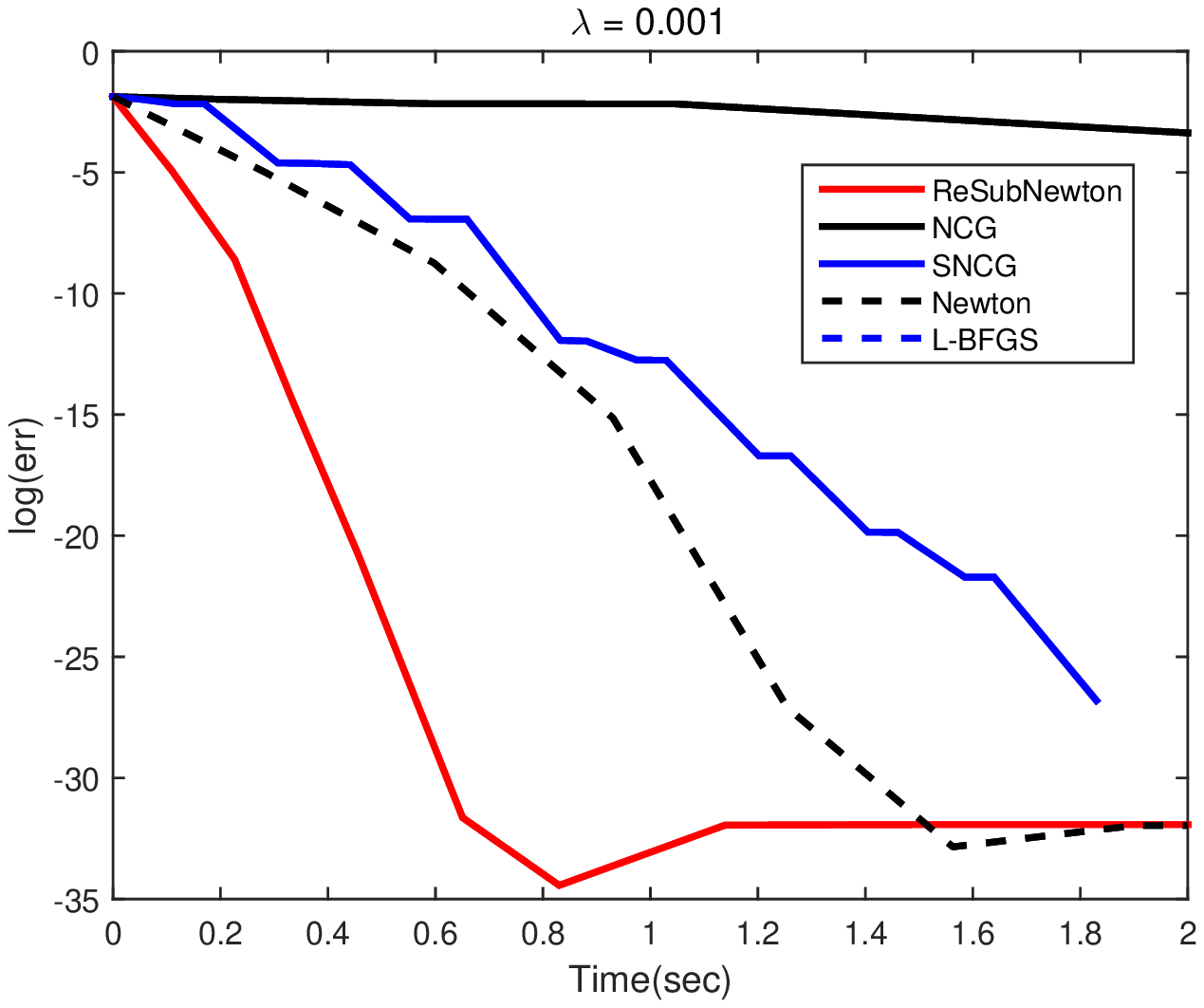}}~
			\subfigure{\includegraphics[width=45mm]{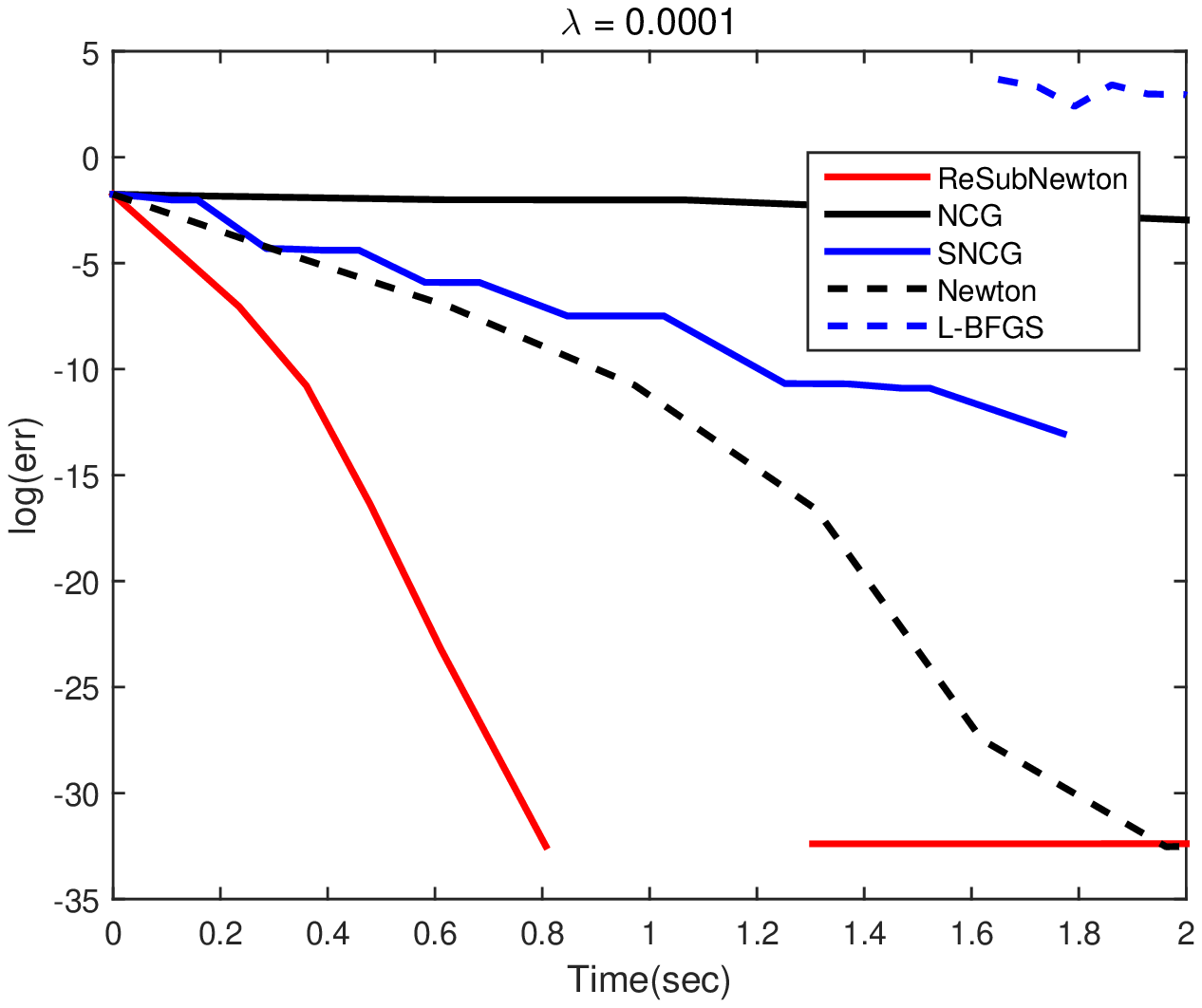}}~
			\subfigure{\includegraphics[width=45mm]{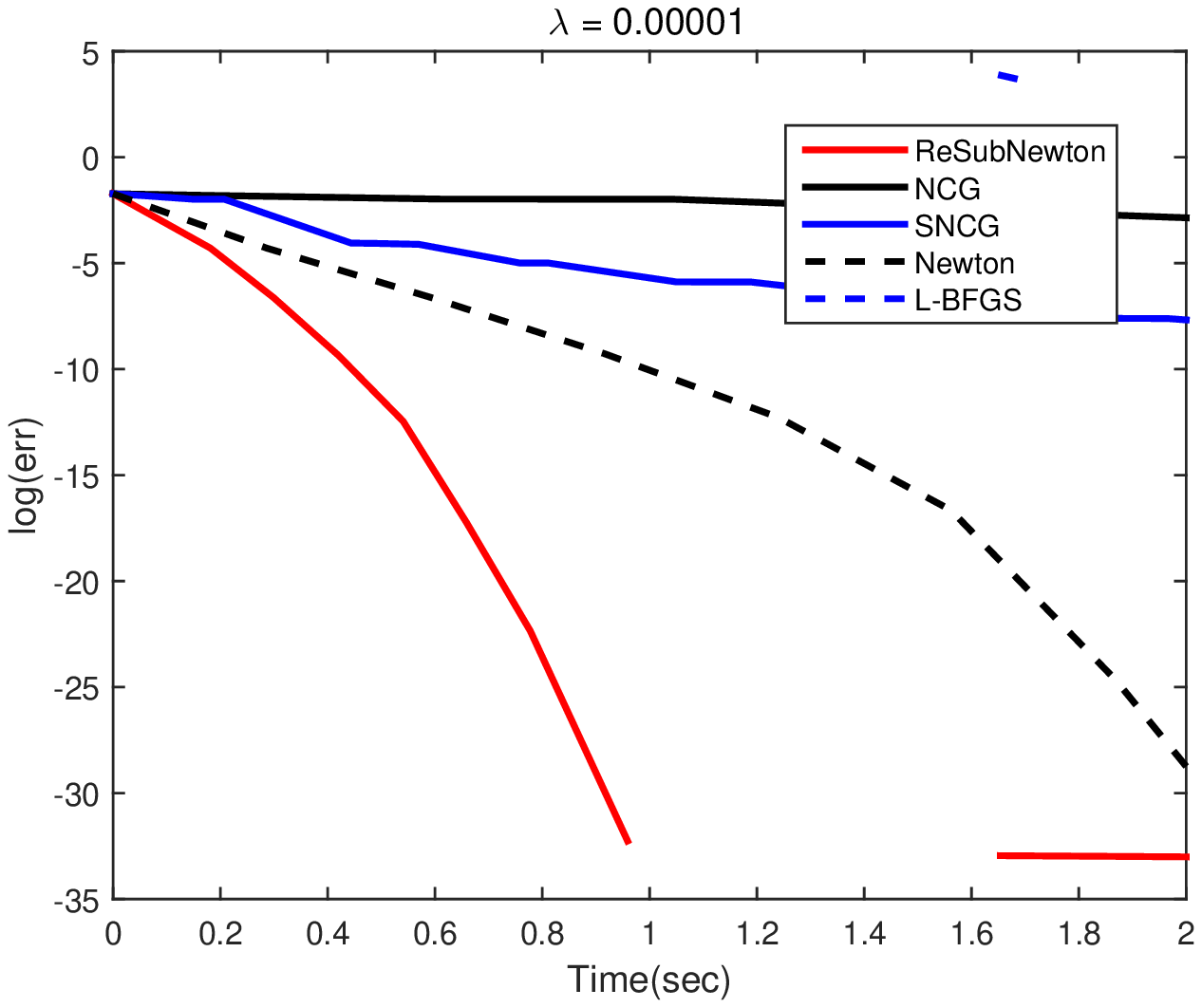}}\\
			\subfigure{\includegraphics[width=45mm]{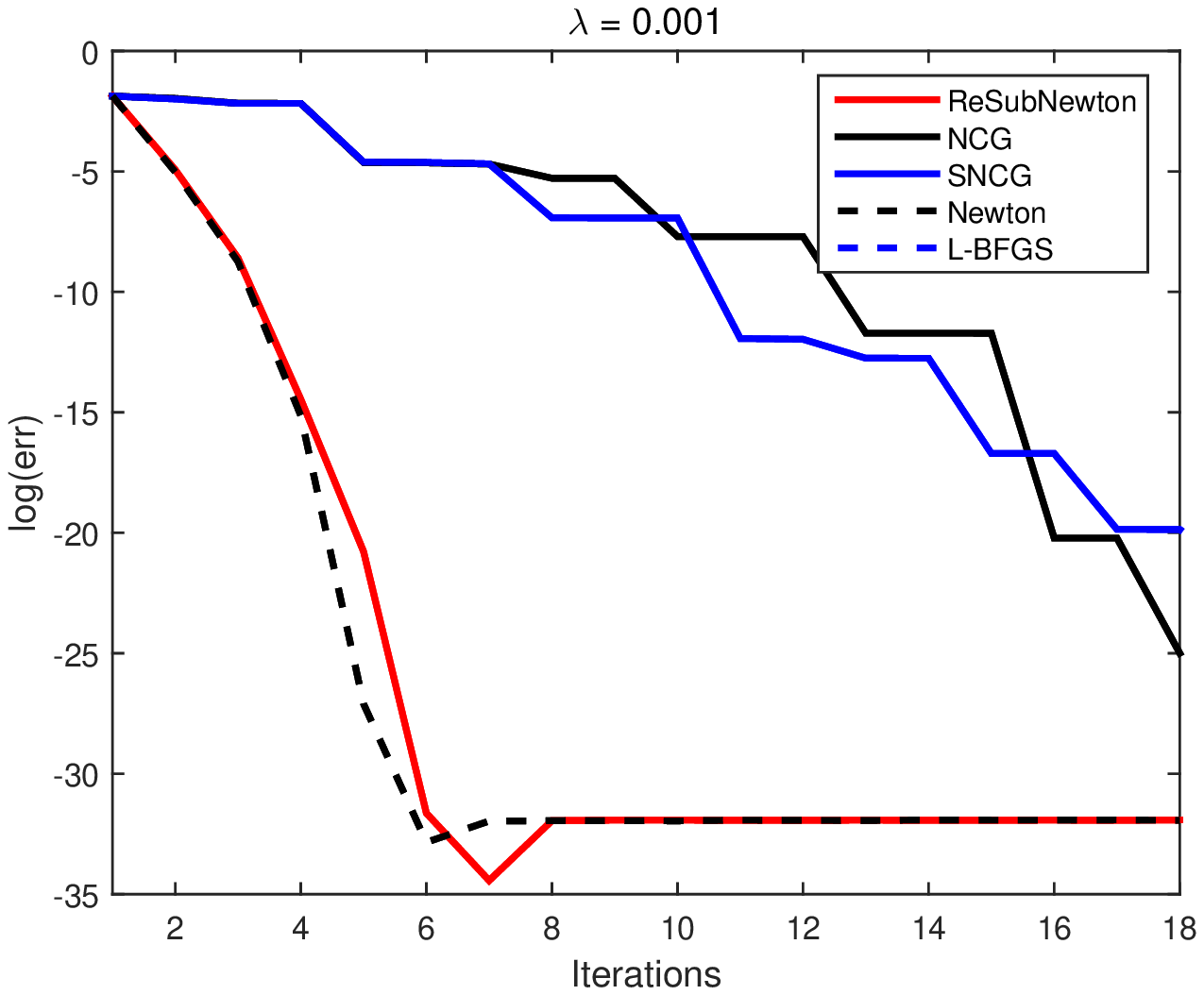}}~
			\subfigure{\includegraphics[width=45mm]{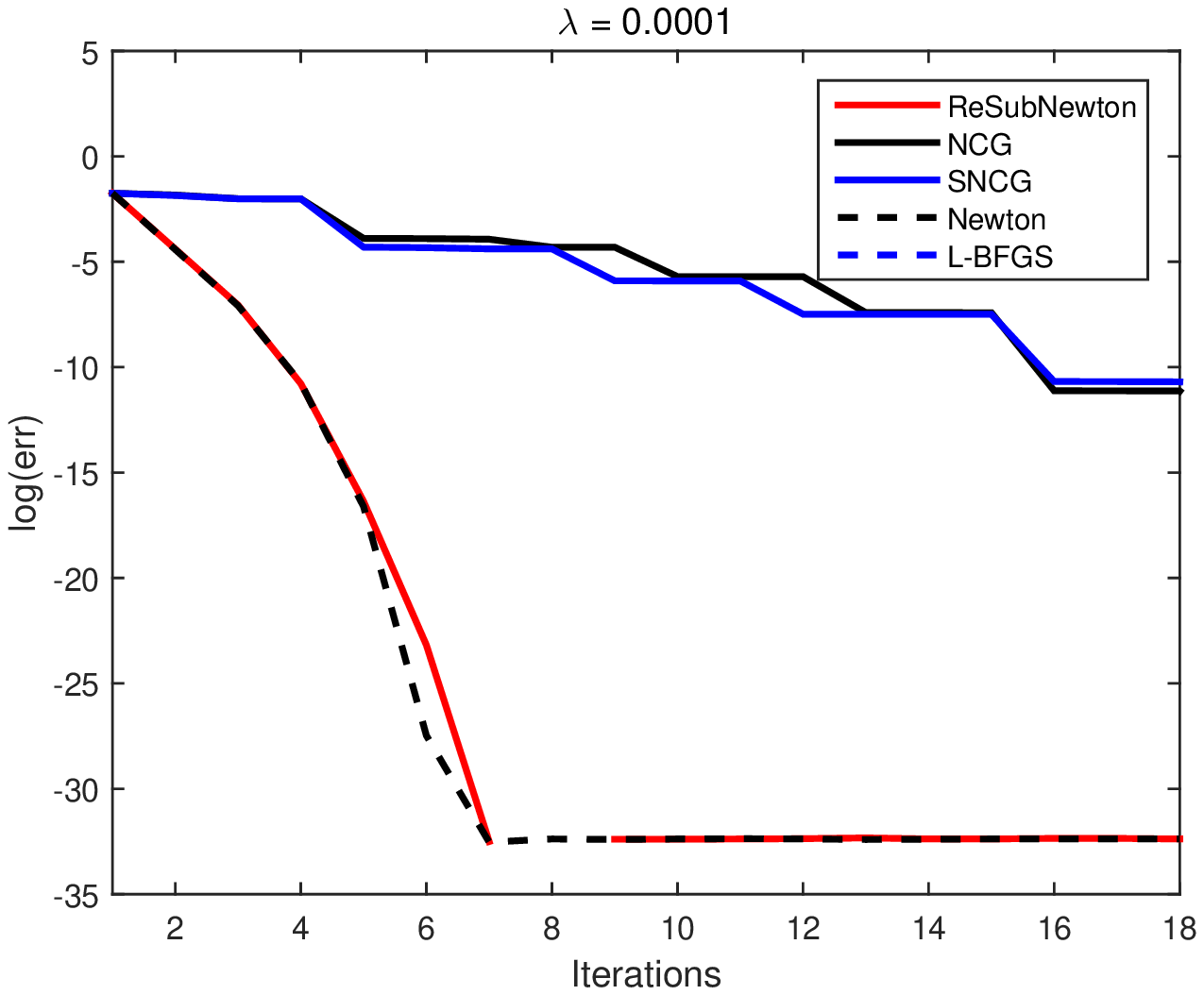}}~
			\subfigure{\includegraphics[width=45mm]{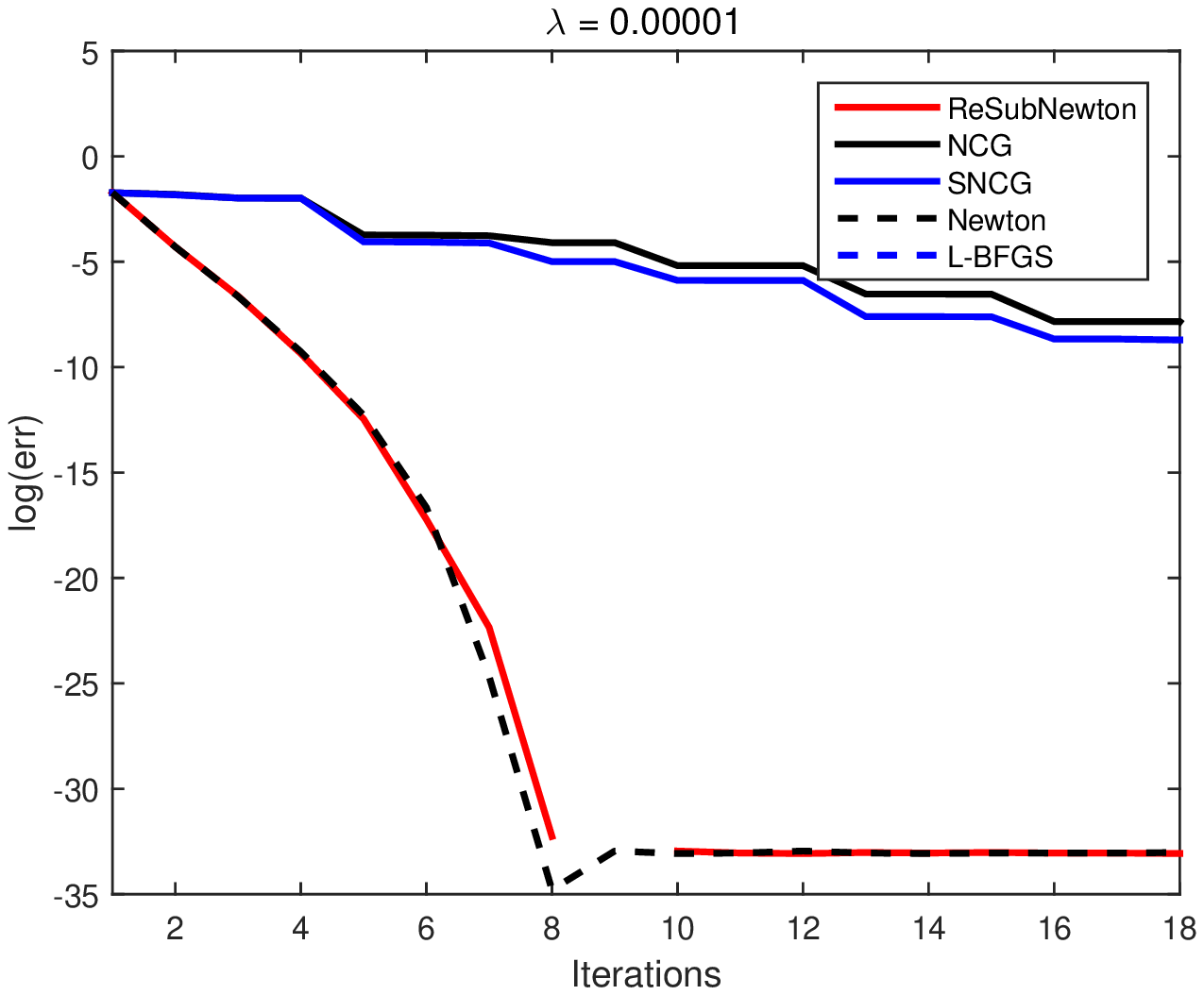}}
		\end{center}
		\caption{Experiment on 'covertype' with different $\lambda$}
		\label{fig:ctp}
	\end{figure}
	\begin{figure}[!ht]
		\subfigtopskip = 0pt
		\begin{center}
			\centering
			\subfigure{\includegraphics[width=45mm]{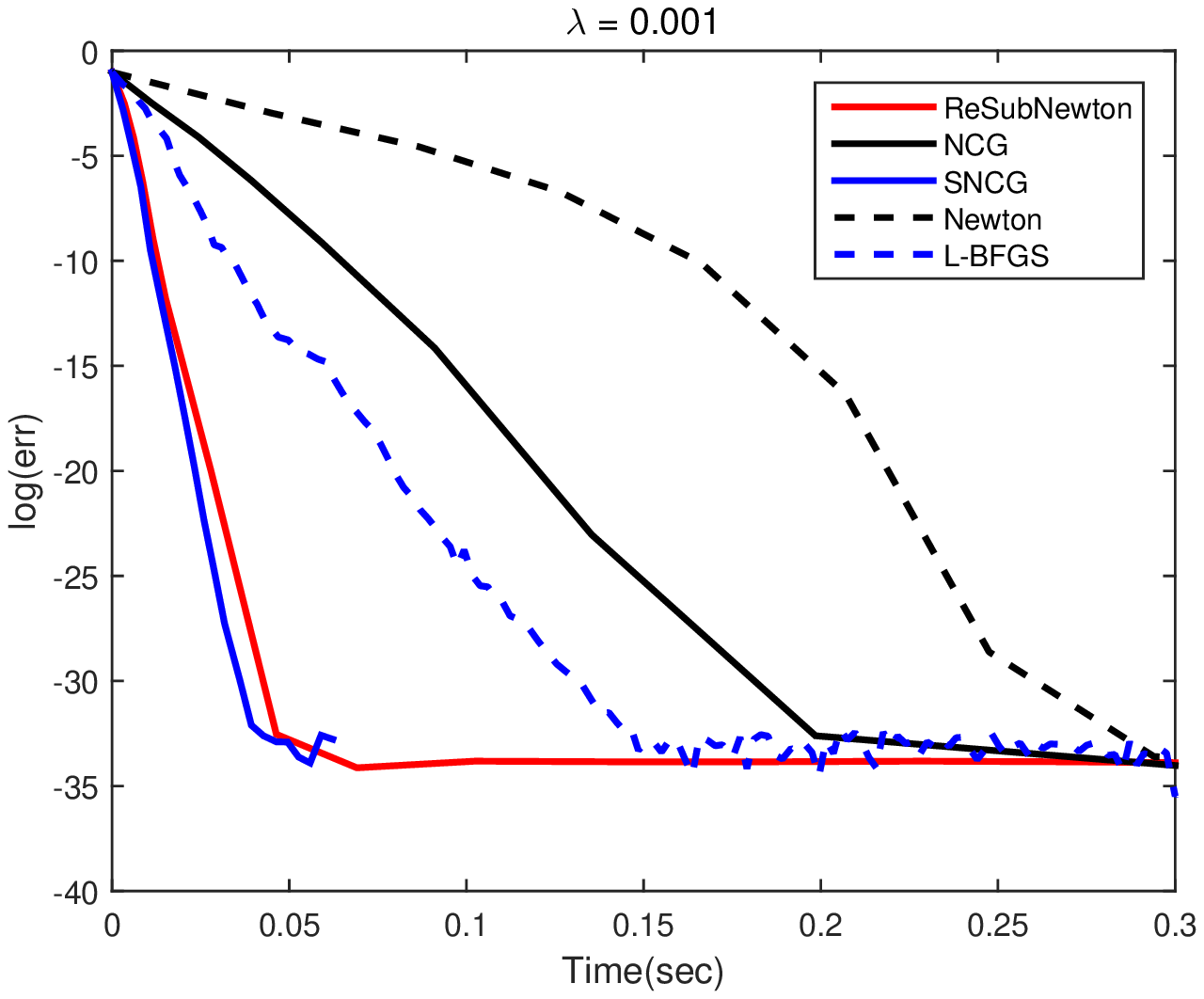}}~
			\subfigure{\includegraphics[width=45mm]{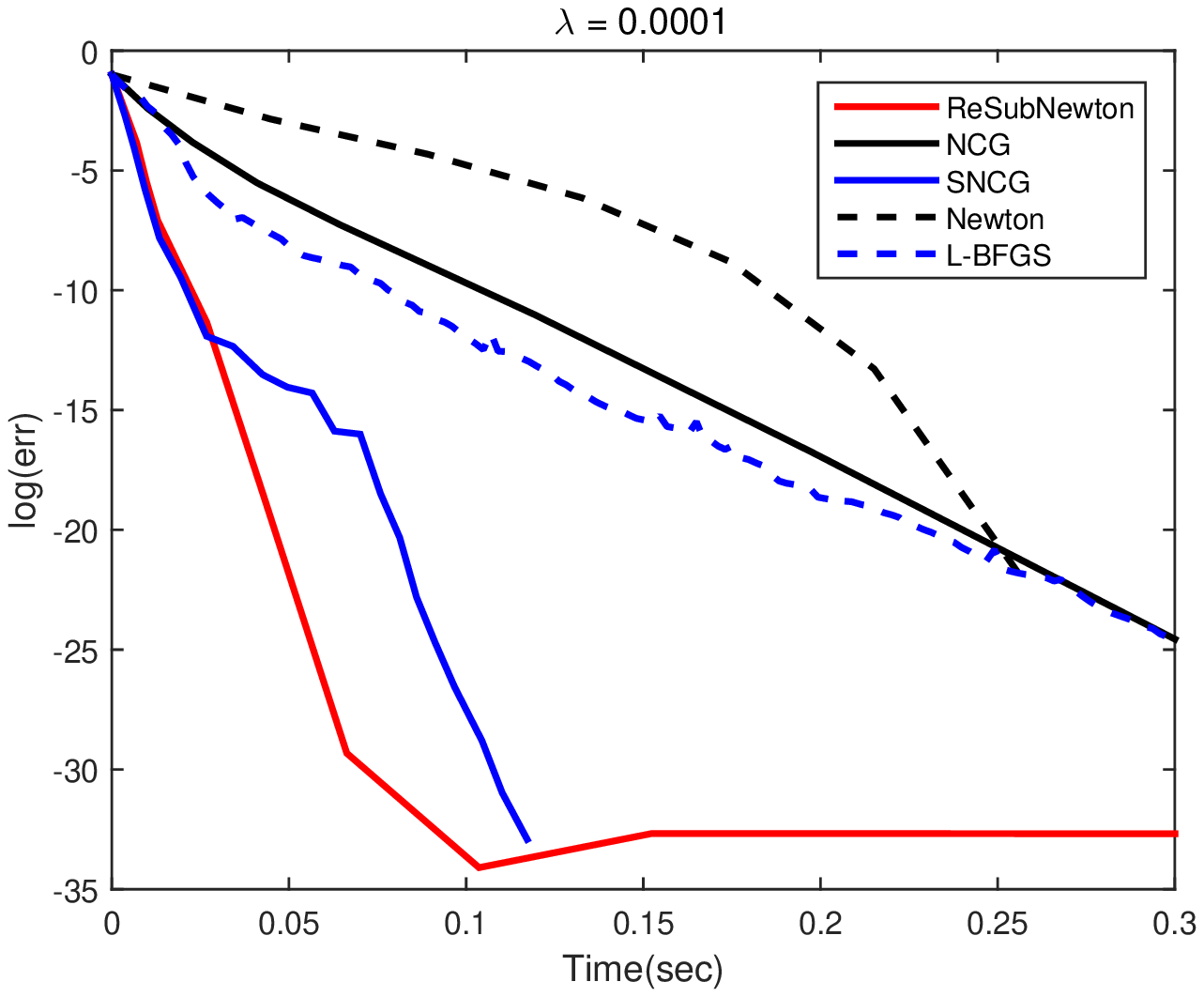}}~
			\subfigure{\includegraphics[width=45mm]{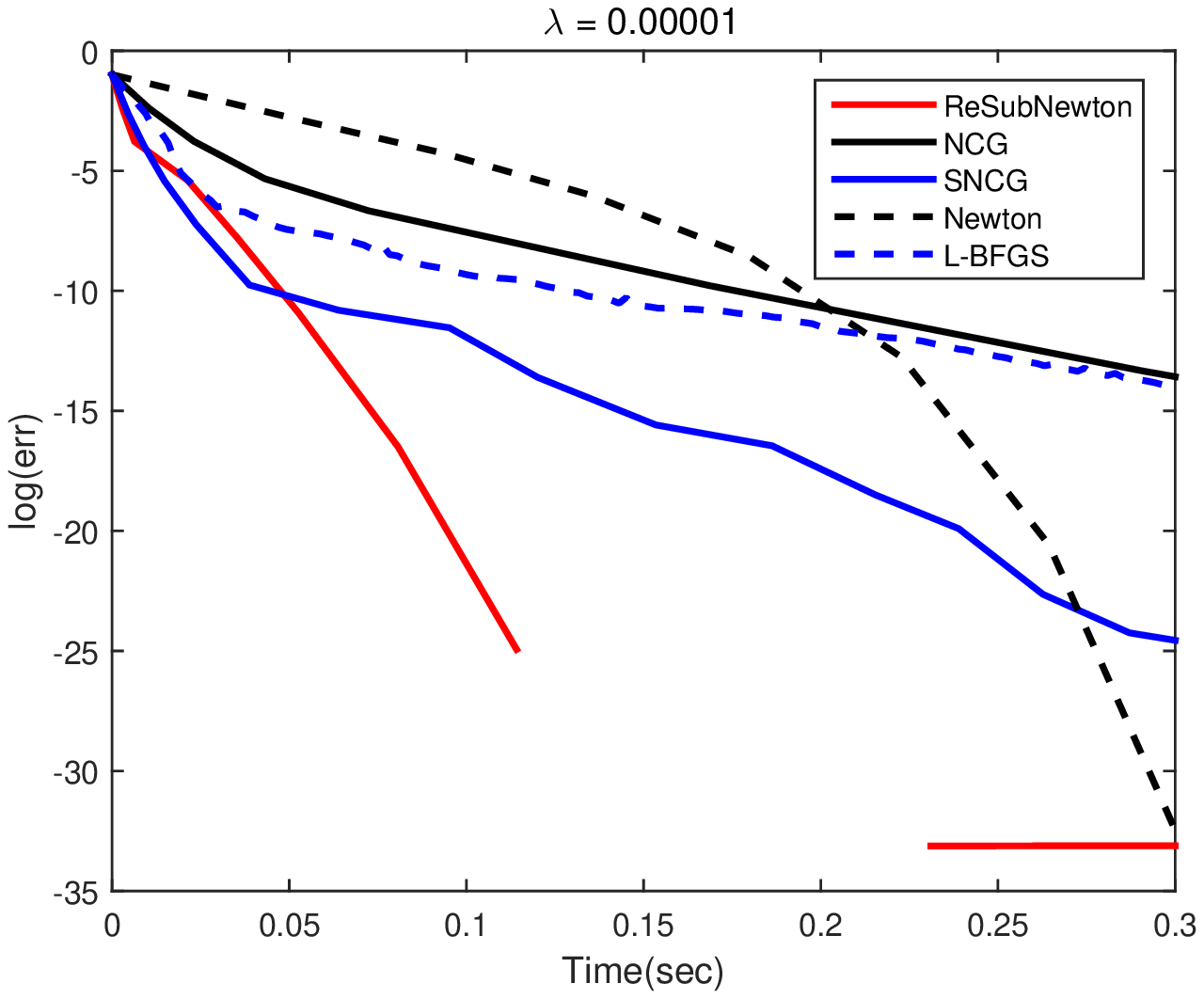}}\\
			\subfigure{\includegraphics[width=45mm]{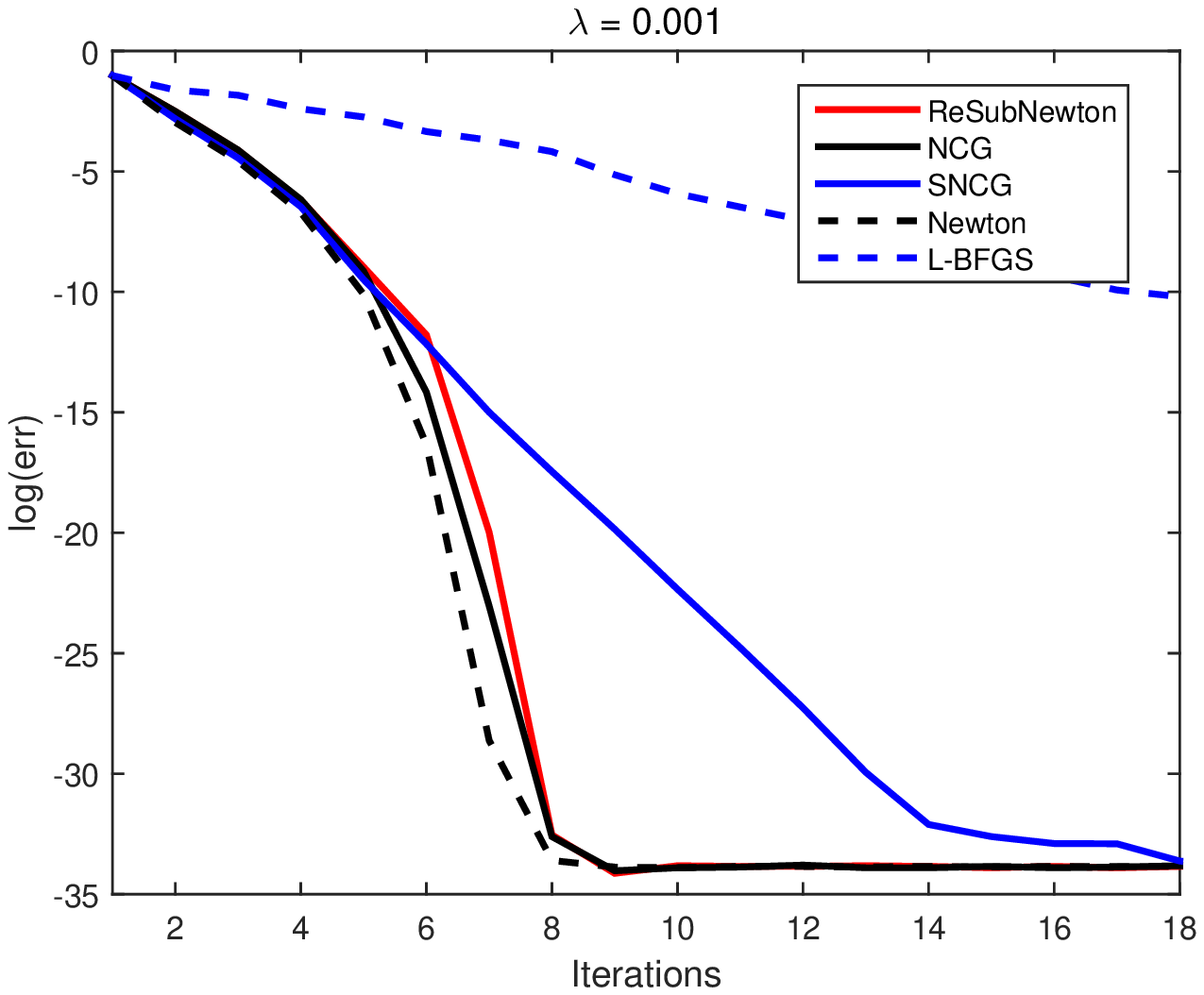}}~
			\subfigure{\includegraphics[width=45mm]{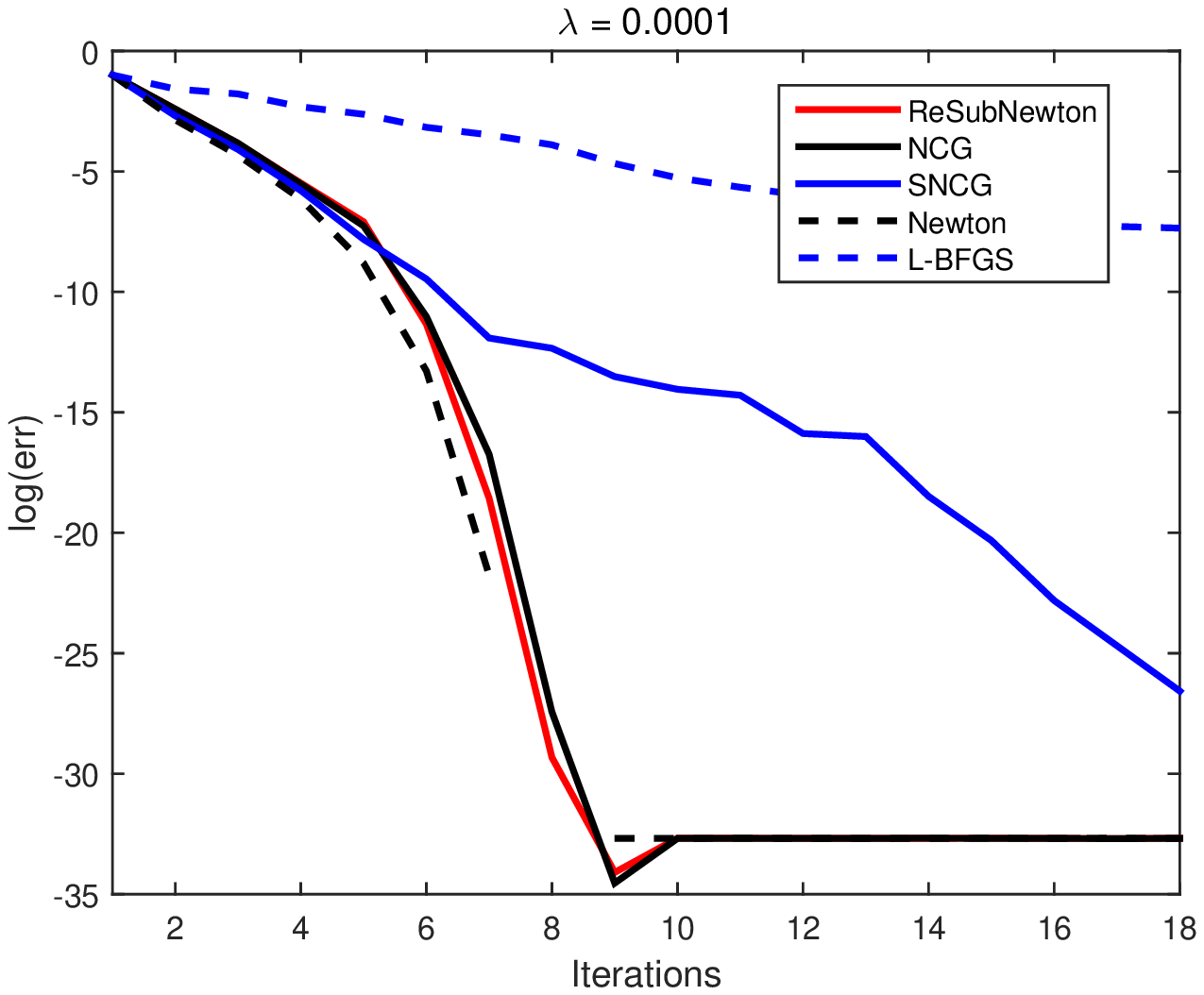}}~
			\subfigure{\includegraphics[width=45mm]{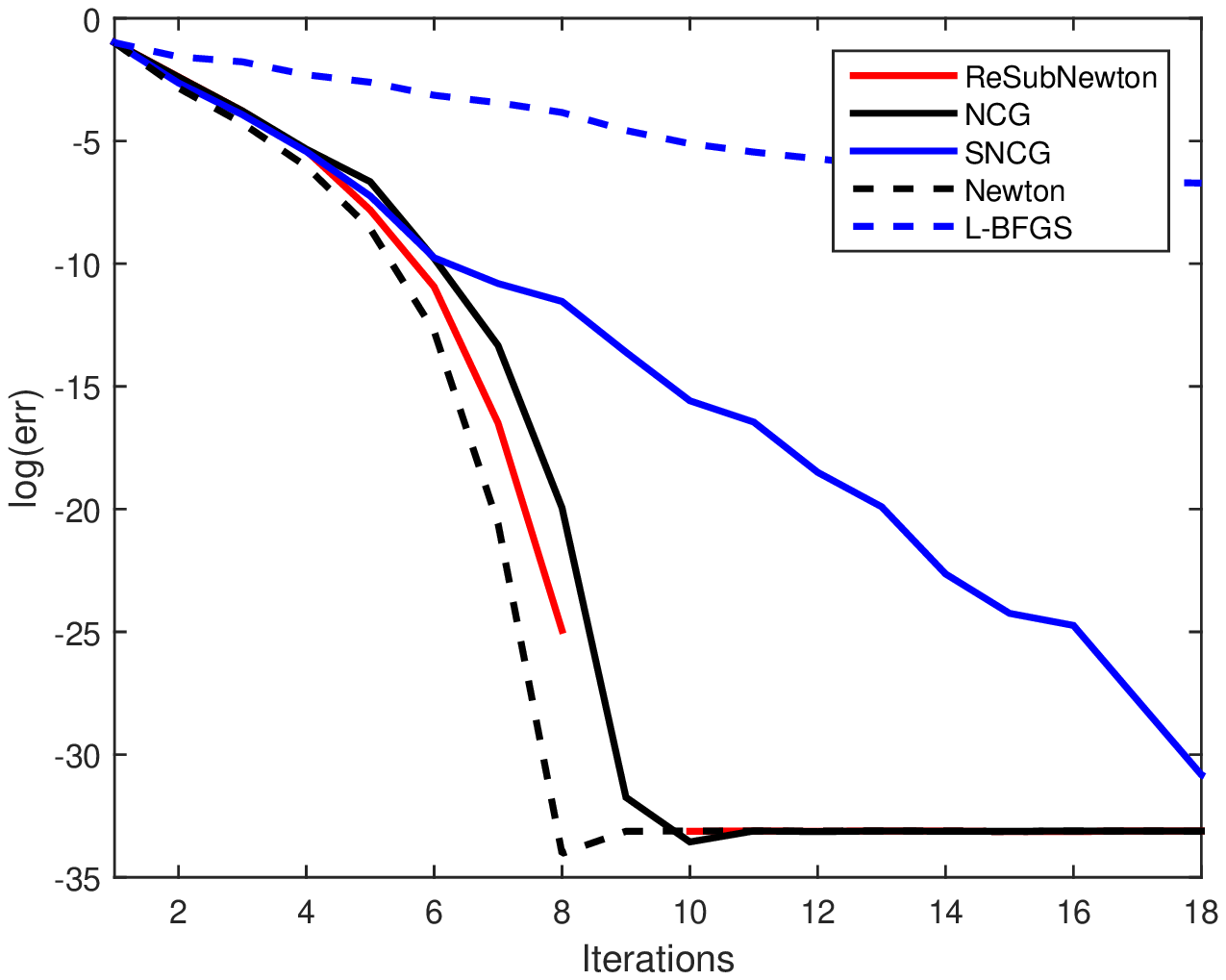}}
		\end{center}
		\caption{Experiment on 'a9a' with different $\lambda$}
		\label{fig:a9a}
	\end{figure}
	\begin{figure}[!ht]
		\subfigtopskip = 0pt
		\begin{center}
			\centering
			\subfigure{\includegraphics[width=45mm]{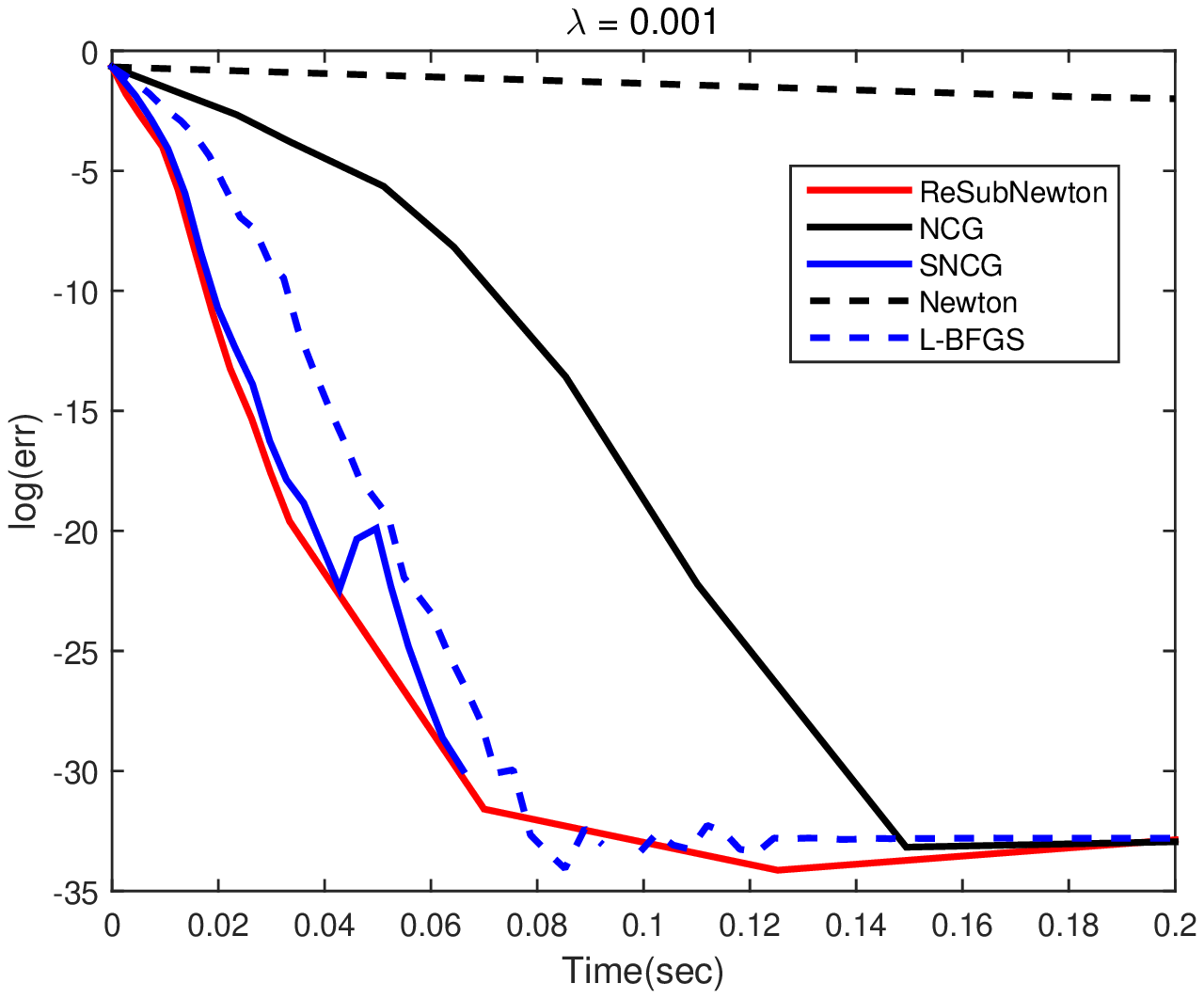}}~
			\subfigure{\includegraphics[width=45mm]{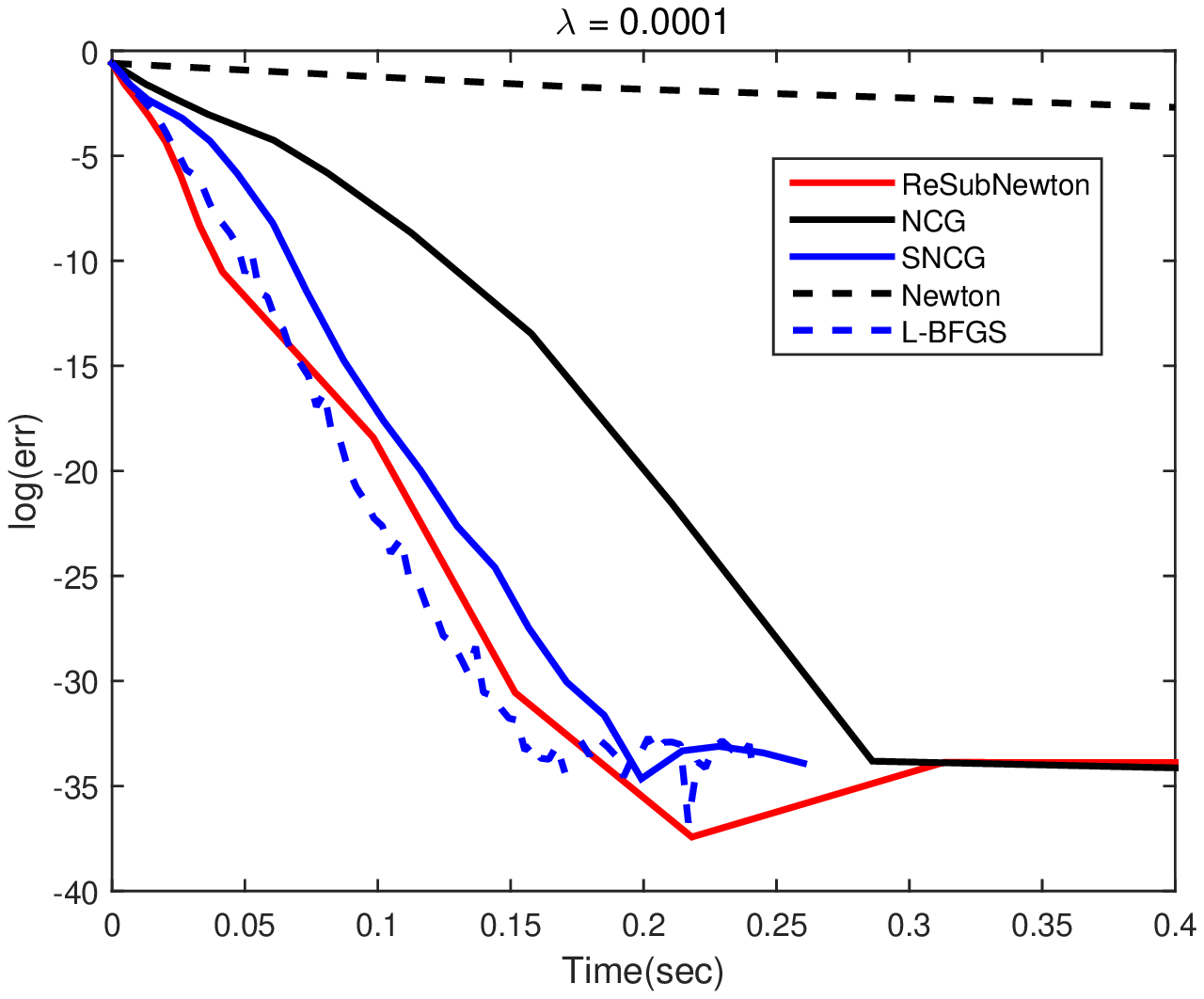}}~
			\subfigure{\includegraphics[width=45mm]{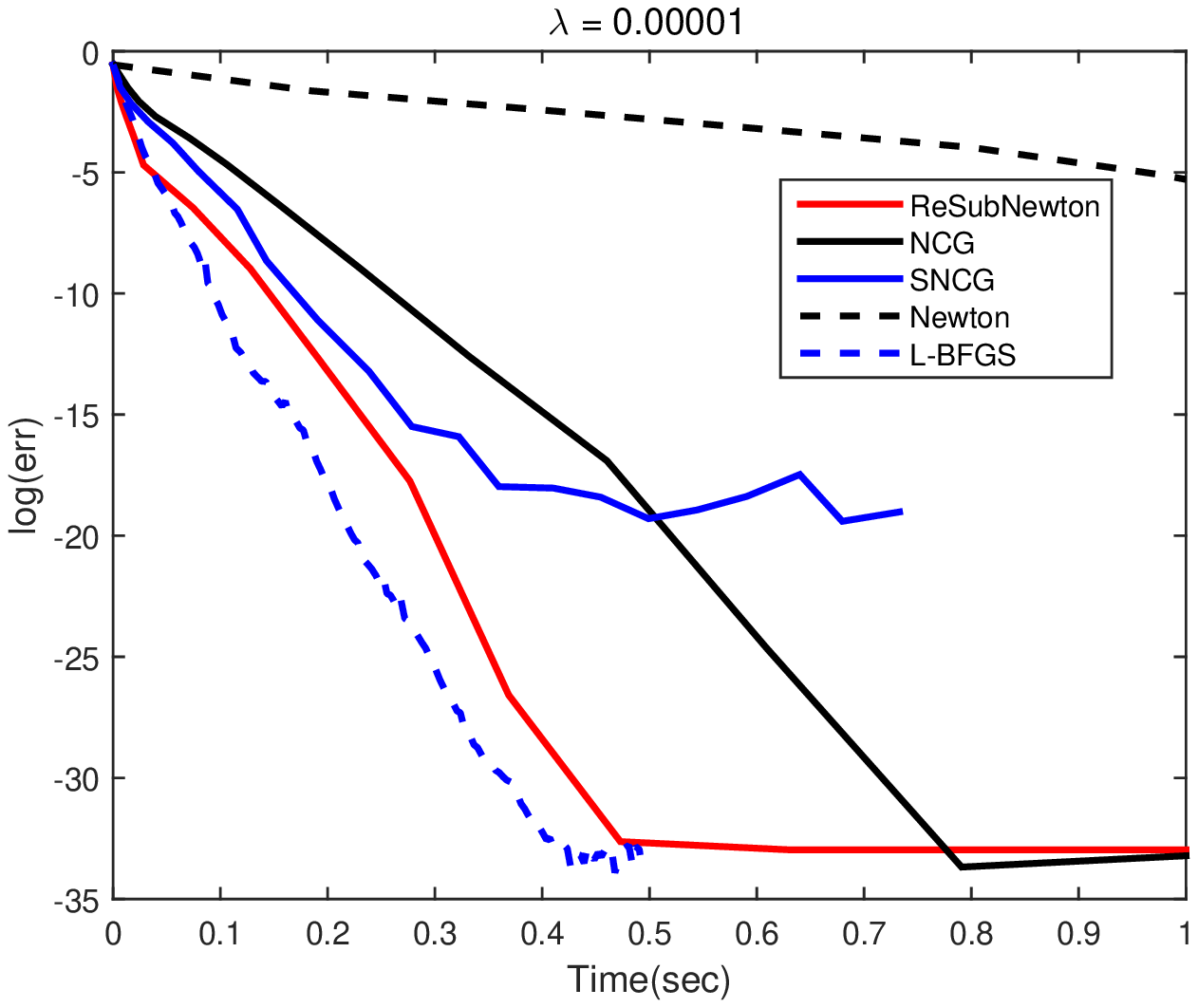}}\\
			\subfigure{\includegraphics[width=45mm]{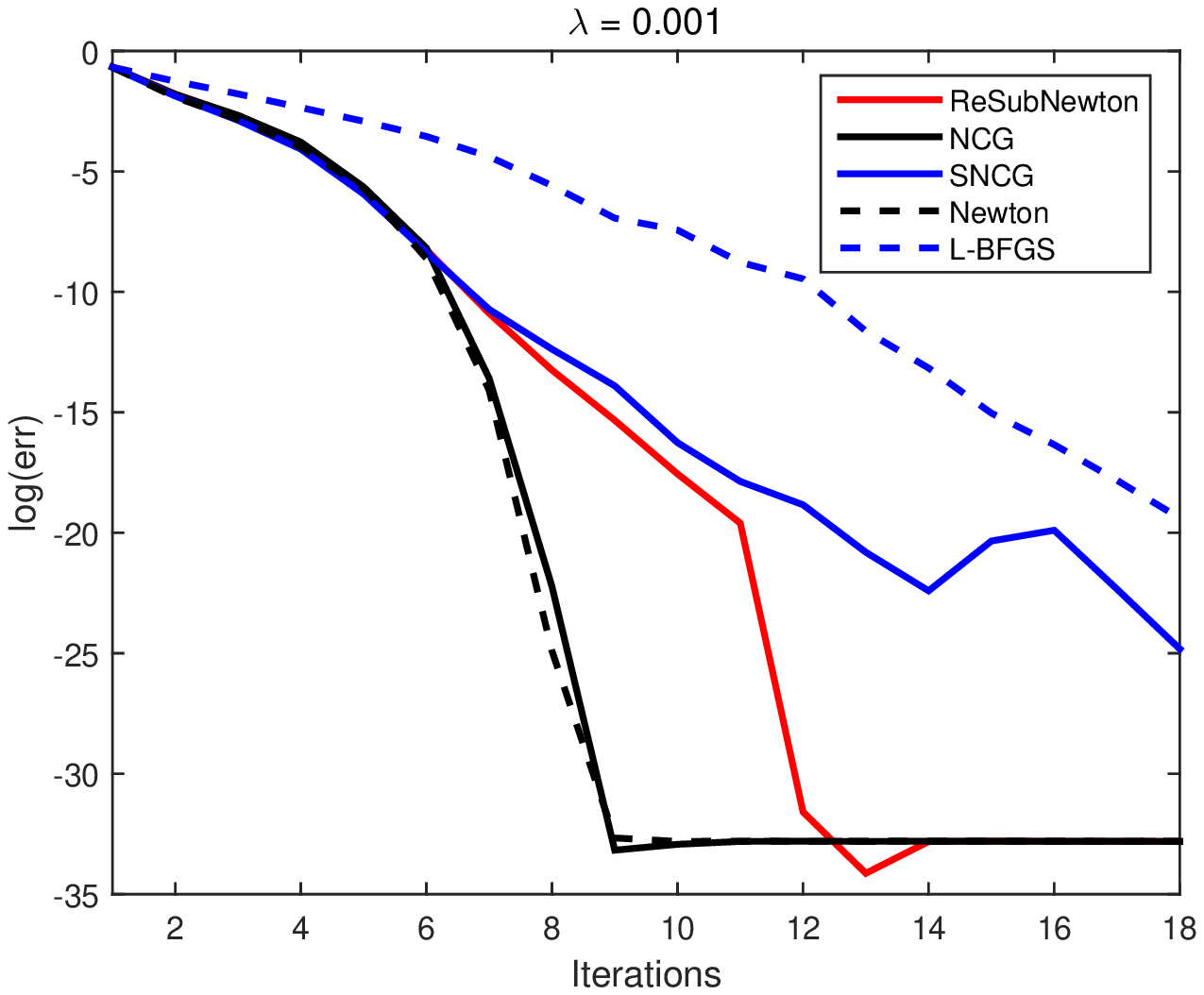}}~
			\subfigure{\includegraphics[width=45mm]{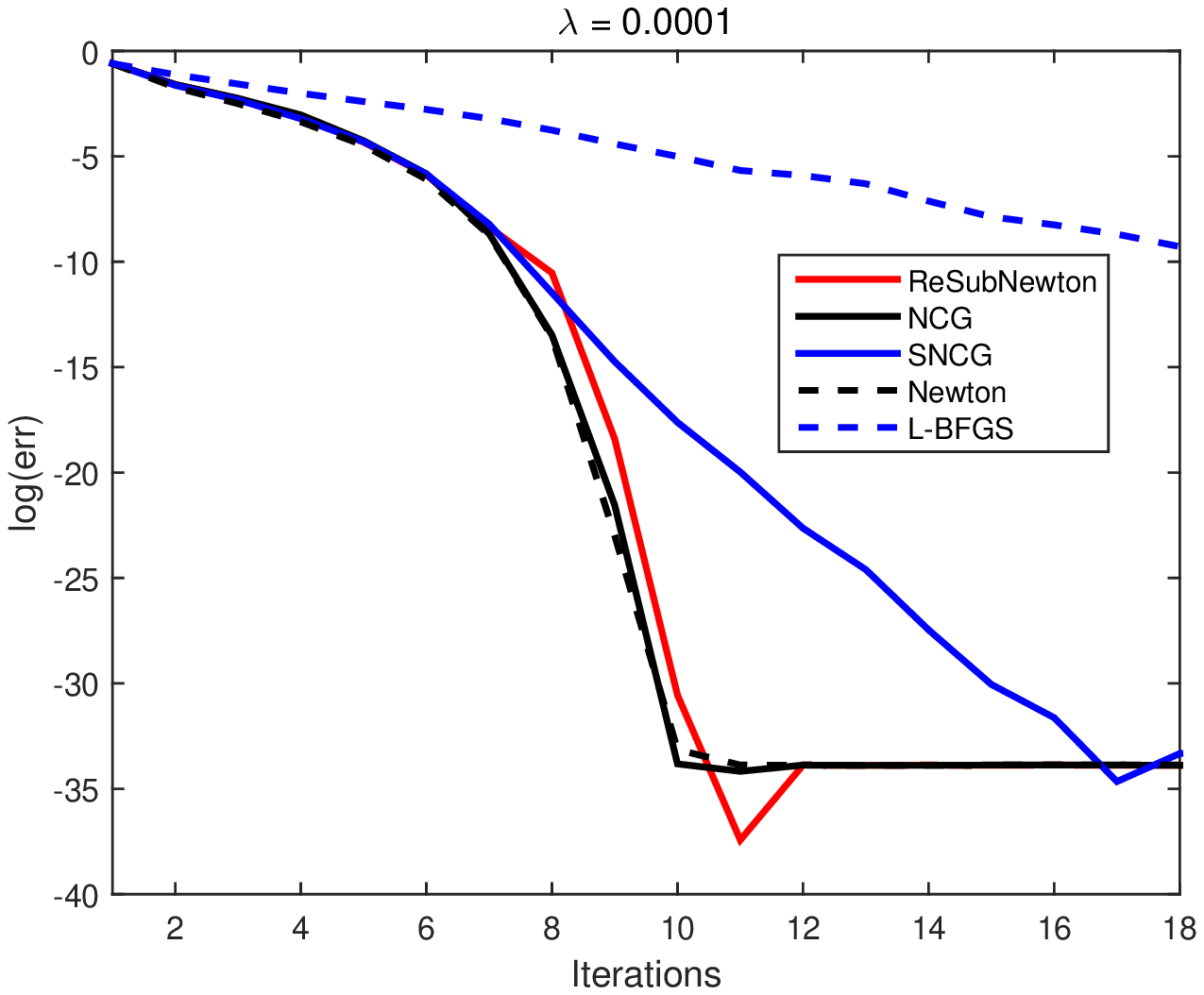}}~
			\subfigure{\includegraphics[width=45mm]{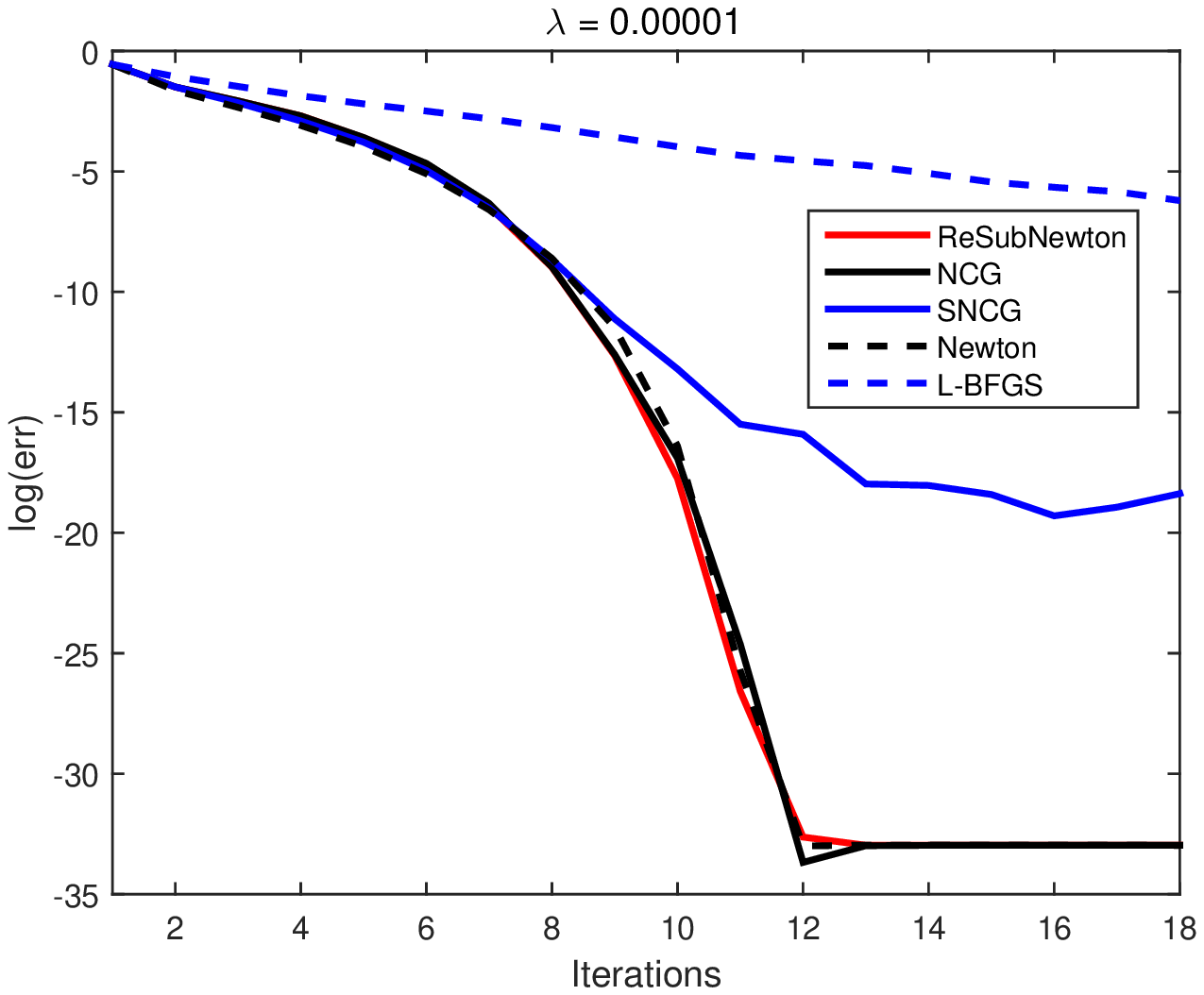}}
		\end{center}
		\caption{Experiment on 'w8a' with different $\lambda$}
		\label{fig:w8a}
	\end{figure}
	
	\section{Conclusion}
	
	In this paper we have proposed two novel sub-sampled Newton methods  called ReSubNewton and ReSkeNewton. They are the first practical sub-sampled Newton method which can achieve superlinear and quadratic convergence rate. We have developed a more general proof framework from a perspective of inexact Newton, which unifies several existing sub-sampled Newton methods. The framework is a fundamental of convergence analysis of sub-sampled Newton methods. Accordingly, we have shown several new convergence properties of sub-sampled Newton methods, which are important both in theory and real application. The empirical studies have validated the efficiency of our algorithms. Our work would be potentially useful for sub-sampled Newton methods.  
	
	\appendix
	
	\section{Some Important Lemmas}
	\begin{lemma} \label{lem:matrix_bnd}
		If \eqref{eq:k} and \eqref{eq:sigma} hold and letting $0<\delta<1$, $0<\epsilon<1$ and $0<c$ be given, besides, the sample size  $|\mathcal{S}| \geq \max({\frac{16K^2\log(2p/\delta)}{c^2 \epsilon^2},\frac{K\log(2p/\delta)}{\sigma \epsilon^2}})$ and $H^{(t)} = \frac{1}{|\SM|}\sum_{j\in\mathcal{S}}\nabla^2 f_j(x^{(t)})$, then we have the following properties:
		\begin{align*}
		&\|H^{(t)} - \nabla^{2}F(x^{(t)})\| \leq \epsilon c,\\
		&\lambda_{\min}(H^{(t)}) \geq (1-\epsilon) \sigma.
		\end{align*}
	\end{lemma}
	\begin{proof}
		Consider $|\SM|$ i.i.d random matrces $H_j^{(t)}, j = 1,\dots,|\SM|$ such that $\PB(H_j^{(t)} = \nabla^{2}f_i(x^{(t)}) = 1/n$ for all $i = 1,\dots,n$. Then, we have $\EB(H_j^{(t)}) = \nabla^{2}F(x^{(t)})$ for all $j = 1,\dots,|\SM|$. By \eqref{eq:k} and the positive semi-definite property of $H_j^{(t)}$, we have $\lambda_{\max}(H_j^{(t)}) \leq K$ and $\lambda_{\min}(H_j^{(t)}) \geq 0$. By Matrix Chernoff bound, we have that if $|\SM| \geq \frac{K\log (p/\delta)}{\sigma\epsilon^{2}}$, $\lambda_{\min}(H^{(t)}) \geq (1-\epsilon) \sigma$ holds with probability at least $1-\delta$.
		
		We define random maxtrices $X_j = H_j^{(t)} - \nabla^{2}F(x^{(t)})$ for all $j = 1,\dots,|\SM|$.	We have $\EB[X_j] = 0$, $\|X_j\| \leq 2K$ and $\|X_j\|^2 \leq 4K^2$. By Matrix Bernstein, we have
		\[
		\PB(\|H^{(t)} - \nabla^{2}F(x^{(t)})\|\geq \epsilon c) \leq 2p\exp^{-\frac{c^2\epsilon^2|\SM|}{16K^2}}.
		\]
		When $|\SM| \geq \frac{16K^2\log(2p/\delta)}{c^2 \epsilon^2}$, $\|H^{(t)} - \nabla^{2}F(x^{(t)})\| \leq \epsilon c$ holds with probability at least $1-\delta$.
	\end{proof}
	
	\begin{lemma}[\cite{Spielman}]\label{lem:refine}
		If $A, B$ are $p\times p$ symmetric nonsingular matrix, and $(1-\epsilon)B \preceq A \preceq (1+\epsilon) B$, where $0<\epsilon<1$, then for the optimization problem $\min_{x}\|Ax-b\|$, we have
		\[
		\|x^{1}-x^{*}\|_{A} \leq \epsilon \|x^{*}\|_A,
		\]
		where $x^{*} = A^{-1}b$ and $x^{1} = B^{-1}b$. Besides, if we set $r^{(t)} = Ax^{(t)} - b$, $x^{(t)}_r = B^{-1}r^{(t)}$ and $x^{(t+1)} = x^{(t)}+x^{(t)}_r$, then $\|x^{(t+1)} -x^{*}\|_A \leq \epsilon^{t+1}\|x^{*}\|_A$.
	\end{lemma}
	\begin{algorithm}[tb]
		\caption{NewSamp.}
		\label{alg:NewSamp}
		\begin{small}
			\begin{algorithmic}[1]
				\STATE {\bf Input:} $x^{(0)}$, $r$, $0<\epsilon<1$, $\{\eta^{(t)}, |\SM^{(t)}|\}$;
				\FOR {$t=0,1,\dots$ until termination }
				\STATE Select a sample set $\SM^{(t)}$, of size $|\SM^{(t)}|$ and get $H_{|\SM^{(t)}|} = \frac{1}{|\SM^{(t)}|}\sum_{j\in\mathcal{S}}\nabla^2 f_j(x^{(t)})$;
				\STATE Compute $r+1$ SVD deompostion of $H_{|\SM^{(t)}|}$ to get $U_{r+1}$ and $\Lambda_{r+1}$. Construct $H^{(t)} = \frac{1}{\eta^{(t)}}(U_r(\Lambda_{r} - \lambda_{r+1}I)U_r^T + \lambda_{r+1}I)$
				\STATE Update $x^{(t+1)}= x^{(t)}-[H^{(t)}]^{-1}\nabla F(x^{(t)})$;
				\ENDFOR
			\end{algorithmic}
		\end{small}
	\end{algorithm}
	
	\begin{algorithm}[tb]
		\caption{Regularized Sub-sample Newton.}
		\label{alg:reg_subsamp}
		\begin{small}
			\begin{algorithmic}[1]
				\STATE {\bf Input:} $x^{(0)}$, $0<\delta<1$, $0<\epsilon<1$, $\alpha$, sample size $|\SM|$ ;
				\FOR {$t=0,1,\dots$ until termination}
				\STATE Select a sample set $\SM$, of size $|\SM|$ and $H^{(t)} = \frac{1}{|\SM|}\sum_{j\in\mathcal{S}}\nabla^2 f_j(x^{(t)}) + \alpha I$;
				\STATE Update $x^{(t+1)}= x^{(t)}-[H^{(t)}]^{-1}\nabla F(x^{(t)})$;
				\ENDFOR
			\end{algorithmic}
		\end{small}
	\end{algorithm}
	
	\section{Proof of Theorem~\ref{thm:inexact_newton}}
	\begin{proof} {\bf{of Theorem~\ref{thm:inexact_newton}}}
		Since $\nabla^2F(x)$ is positive definite at $x^*$, it has
		\begin{align}
		\frac{1}{M}\|y\|\leq\|y\|_*\leq M\|y\|,\text{ for } y\in\RB^n, \label{eq:M_side}
		\end{align}
		where 
		\[
		M \equiv\max(\|\nabla^{2}F(x^*)\|,\|\nabla^{2}F(x^*)^{-1}\|).
		\]
		Because $\nabla^2F(x)$ is continuous near $x^*$, it holds that
		\begin{equation}
		\|[\nabla^{2}F(x^*)]^{-1} - [\nabla^{2}F(x)]^{-1}\| < \varepsilon, \label{eq:inv_close}
		\end{equation}
		\begin{equation}
		\|\nabla^{2}F(x^*) - \nabla^{2}F(x)\| < \eta \label{eq:eta_def}
		\end{equation}
		and
		\begin{equation*}
		\|\nabla F(x^{(t)}) - \nabla F(x^{*}) - \nabla^2F(x^{*})(x^{(t)} - x^*)\| \leq \eta \|x^{(t)} - x^*\|. \label{eq:tmp}
		\end{equation*}
		if $\|x-x^*\|\leq \delta$, where $\varepsilon = o(1)$ and $\eta = o(1)$. The equation~\eqref{eq:tmp} is equivalent to 
		\begin{equation}
		\|\nabla F(x^{(t)}) - \nabla^2F(x^{*})(x^{(t)} - x^*)\| \leq \eta \|x^{(t)} - x^*\|. \label{eq:tl}
		\end{equation}
		since $\nabla F(x^{*}) = 0$.
		
		By \eqref{eq:inv_close}, we can assume that 
		\begin{equation}
		\| [\nabla^{2}F(x^{t})]^{-1}\| \leq M \label{eq:M_def}
		\end{equation}
		for all $x^{(t)}$ sufficiently close to $x^*$.
		Therefore, we have from \eqref{eq:res_def} that the inexact Newton step satisfies
		\begin{equation}
		\|v^{(t)}\| = \|[\nabla^{2}F(x^{t})]^{-1}\|(\|r^{(t)}\|+\|\nabla F(x^{(t)})\|)  \leq 2M\|\nabla F(x^{(t)})\|
		\end{equation}
		where the second inequality is because $\|r^{(t)}\| \leq \gamma^{(t)}\|\nabla F(x^{(t)})\|$ and $\gamma^{(t)}<1$. Combining Taylor's theorem and the continuity of $\|\nabla^{2}F(x^*)\|$, we obtain
		\begin{align}
		{}&\|\nabla F(x^{(t+1)})\| \notag\\
		={}& \|\nabla F(x^{(t)})+\nabla^2 F(x^{(t)})(-v^{(t)}) + \int_0^1 [\nabla^2 F(x^{(t)} + sv^{(t)}) - \nabla^2 F(x^{(t)})](-v^{(t)})ds\| \notag \\
		\leq{}&\|\nabla F(x^{(t)})-\nabla^2 F(x^{(t)})v^{(t)}\| + \int_0^1 \|\nabla^2 F(x^{(t)} + sv^{(t)}) - \nabla^2 F(x^{(t)})\|ds\|v^{(t)}\| \notag\\
		\leq{}&\|r^{(t)}\| + 2\eta\|v^{(t)}\| \label{eq:conti} \\
		\leq{}& \gamma^{(t)}\|\nabla F(x^{(t)})\| + 2M\eta\|\nabla F(x^{(t)})\| \label{eq:M_mt}\\
		={}& (\gamma^{(t)}+ 2M\eta) \|\nabla F(x^{(t)})\|\notag.
		\end{align}
		Inequality~\eqref{eq:M_mt} follows the definition of $r^{(t)}$ and~\eqref{eq:M_def}. And inequality \eqref{eq:conti} is because $\nabla^2 F(x^{(t)})$ is continuous near $x^*$. If $\|x-x^*\|\leq \delta$, we have 
		\begin{align*}
		{}&\int_0^1 \|\nabla^2 F(x^{(t)} + sv^{(t)}) - \nabla^2 F(x^{(t)})]\|ds \\ 
		={}& \int_0^1 \|\nabla^2 F(x^{(t)} + sv^{(t)}) - \nabla^2 F(x^{*})+\nabla^2 F(x^{*}) - \nabla^2 F(x^{(t)})\|ds \\
		\leq{}&\int_0^1 [\|\nabla^2 F(x^{(t)} + sv^{(t)}) - \nabla^2 F(x^{*})\|+\|\nabla^2 F(x^{*}) - \nabla^2 F(x^{(t)})\|]ds \\
		\leq{}&\int_0^1[\eta + \eta]ds = 2 \eta.
		\end{align*}
		
		We define $\|y\|_* \equiv \|\nabla^2F(x^{*})y\|$. Therefore, we have
		\begin{align}
		\|x^{(t+1)} - x^{*}\|_* 
		\leq{}& \|\nabla F(x^{(t+1)})\| + \eta\|x^{(t+1)} - x^*\| \label{eq:aa}\\
		\leq{}& (\gamma^{(t)}+ 2M\eta) \|\nabla F(x^{(t)})\| + \eta\|x^{(t)}-v^{(t)} - x^*\|\notag\\
		\leq{}& (\gamma^{(t)}+ 2M\eta) \|\nabla F(x^{(t)})\| + \eta\|x^{(t)} - x^* \| + 2M\eta\|\nabla F(x^{(t)})\| \notag\\
		\leq{}&(\gamma^{(t)}+ 4M\eta)(\|x^{(t)} - x^* \|_*+\eta\|x^{(t)} - x^* \|) + \eta\|x^{(t)} - x^* \| \label{eq:bb}\\
		\leq{}&(\gamma^{(t)}+ 4M\eta)(\|x^{(t)} - x^* \|_*+M\eta\|x^{(t)} - x^* \|_*) + M\eta\|x^{(t)} - x^* \|_* \label{eq:cc}\\
		\leq{}& (\gamma^{(t)} + 6M\eta + 4M^2\eta^2)\|x^{(t)} - x^* \|_* \notag\\
		={}&(\gamma^{(t)} + 6M\eta) \|x^{(t)} - x^* \|_*. \label{eq:dd}
		\end{align}
		Equation~\eqref{eq:aa} and~\eqref{eq:bb} follow from~\eqref{eq:tl} and inequality~\eqref{eq:cc} is because of \eqref{eq:M_side}. Equation~\eqref{eq:dd} just omits $4M^2\eta^2$ since $\eta = o(1)$.
		
		Hence, we obtain 
		\[
		\|x^{(t+1)} - x^{*}\|_* \leq (\gamma^{(t)} + 6M\eta)\|x^{(t)} - x^* \|_*.
		\]
		
		The proof is similar when $\nabla^2 F(x)$ is Lipschitz continous near $x^*$ with parameter $L$. We have
		\begin{equation}
		\|\nabla^2F(x) - \nabla^2F(x^{*})\| \leq L\|x-x^*\| \label{eq:L-1}
		\end{equation} 
		and
		\begin{equation}
		\|\nabla F(x) - \nabla^2F(x^{*})(x - x^*)\| \leq L \|x - x^*\|^2,  \label{eq:L-2}
		\end{equation}
		when $x$ is sufficiently close to $x^*$.
		
		Then, combining Taylor's theorem, we obtain
		\begin{align}
		{}&\|\nabla F(x^{(t+1)})\| \notag\\
		={}& \|\nabla F(x^{(t)})+\nabla^2 F(x^{(t)})(-v^{(t)}) + \int_0^1 [\nabla^2 F(x^{(t)} + sv^{(t)}) - \nabla^2 F(x^{(t)})](-v^{(t)})ds\| \notag \\
		\leq{}&\|r^{(t)}\| + L\|v^{(t)}\|^2  \label{eq:L-3}\\
		\leq{}& \gamma^{(t)}\|\nabla F(x^{(t)})\| + 4LM^2\|\nabla F(x^{(t)})\|^2 \notag
		\end{align}
		where, inequality~\eqref{eq:L-3} follows from \eqref{eq:L-1}.
		Also, by~\eqref{eq:L-2}, we have
		\begin{align}
		\|x^{(t+1)} - x^{*}\|_* 
		\leq{}& \|\nabla F(x^{(t+1)})\| + L\|x^{(t+1)} - x^*\|^2 \notag \\
		\leq{}& \|\nabla F(x^{(t+1)})\| +L(M\|x^{(t)}-x^*\|_* + 2M\|\nabla F(x^{(t)})\|)^2\notag\\
		\leq{}& \gamma^{(t)}\|x^{(t)}-x^*\|_*+ 6LM^2\|x^{(t)}-x^*\|_*^2 + 8LM^2\|\nabla F(x^{(t)})\|^2+o(\|x^{(t)}-x^*\|_*^2)\notag\\
		\leq{}& \gamma^{(t)}\|x^{(t)}-x^*\|_*+ 14LM^2\|x^{(t)}-x^*\|_*^2 + o(\|x^{(t)}-x^*\|_*^2) \notag\\
		={}&\gamma^{(t)}\|x^{(t)}-x^*\|_*+ 14LM^2\|x^{(t)}-x^*\|_*^2 \notag
		\end{align}
		Therefore, we obtain
		\begin{equation}
		\|x^{(t+1)} - x^{*}\|_*  \leq \gamma^{(t)}\|x^{(t)}-x^*\|_*+ 14LM^2\|x^{(t)}-x^*\|_*^2.
		\end{equation}
	\end{proof}
	
	\section{Proofs of theorems of Section~\ref{sec:frame}} \label{sec:app_proof}
	\begin{proof} {\bf{of Theorem~\ref{thm:univ_frm}}}\\
		We have 
		\begin{align*}
		\|r^{(t)}\| ={}& \|\nabla^2F(x^{(t)})v^{(t)} + \nabla F(x^{t}) \|\\
		={}& \|(-\nabla^2F(x^{(t)})[H^{(t)}]^{-1} + I)\| \cdot \|\nabla F(x^{t})\| \\
		={}& \|\nabla^2F(x^{(t)}) ([H^{(t)}]^{-1} - [\nabla^2F(x^{(t)})]^{-1})\| \cdot\|\nabla F(x^{(t)})\| \\
		={}& \|\nabla^2F(x^{(t)}) [\nabla^2F(x^{(t)})]^{-1}(\nabla^2F(x^{(t)}) - H^{(t)}) [H^{(t)}]^{-1}\| \cdot\|\nabla F(x^{(t)})\| \\
		={}& \|(\nabla^2F(x^{(t)}) - H^{(t)}) [H^{(t)}]^{-1}\|\cdot\|\nabla F(x^{(t)})\|
		\end{align*}
		For convergence rate analysis, the first convergence rate result can derived directly from Equation~\eqref{eq:conv_lin}. The second one follows from Equation~\eqref{eq:conv_lin} and $\frac{\|x^{(t+1)} - x^{*}\|_*}{\|x^{(t)} - x^{*}\|_*}\to 0$ when $\gamma^{(t)} \to 0$,  and $\eta = o(1)$. For the third convergence result, Equation~\eqref{eq:Lip_conv} leads to the convergence rate and shows that when $\|\nabla F(x^{(t)})\|$ is big enough that $\gamma^{(t)} = \OM(\|\nabla F(x^{(t)}))$ holds, sequence $\{x^{(t)}\}$ will start with a quadratic rate of convergence. However, when $\|\nabla F(x^{(t)})\|$ will decrease to a small value, and $\gamma^{(t)} = \OM(\|\nabla F(x^{(t)}))$ will not hold any more which leads to a linear convergence rate. The forth one is because $\frac{\|x^{(t+1)} - x^{*}\|_*}{\|x^{(t)} - x^{*}\|_*^2} = \OM(1)$ when $\gamma^{(t)} = \OM(\|\nabla F(x^{t})\|)$ by Equation~\eqref{eq:Lip_conv}.
	\end{proof}
	
	\begin{proof} {\bf{of Theorem~\ref{thm:NewSamp}}}\\
		We have
		\begin{align*}
		{}&\|(\nabla^2F(x^{(t)}) - H^{(t)}) [H^{(t)}]^{-1}\| \\
		={}& \|(\nabla^2F(x^{(t)}) -H_{|\SM^{(t)}|}+H_{|\SM^{(t)}|}- H^{(t)}) [H^{(t)}]^{-1}\| \\
		\leq{}&\|(\nabla^2F(x^{(t)}) -H_{|\SM^{(t)}|}\|\cdot\|[H^{(t)}]^{-1}\| + \|I-H_{|\SM^{(t)}|}[H^{(t)}]^{-1}\|
		\end{align*}
		
		By Lemma~\ref{lem:matrix_bnd} with $c = \lambda_{r+1}^{(t)}$ and $\|[H^{(t)}]^{-1}\| = \eta^{(t)}/\lambda_{r+1}^{t}$ , with probability at least $1-\delta$, it holds that
		\[
		\|(\nabla^2F(x^{(t)}) -H_{|\SM^{(t)}|}\|\cdot\|[H^{(t)}]^{-1}\| \leq \eta^{(t)}\frac{4K}{\lambda_{r+1}^{(t)}}\sqrt{\frac{\log (2p/\delta)}{|\SM^{(t)}|}}.
		\]
		Besises, we have
		\[
		\|I-H_{|\SM^{(t)}|}[H^{(t)}]^{-1}\| = 1-\eta^{(t)}\frac{\lambda_p^{(t)}}{\lambda^{(t)}_{r+1}}.
		\]
		Hence, $\gamma^{(t)}$ has the following upper bound:
		\[
		\gamma^{(t)} \leq 1-\eta^{(t)}\frac{\lambda_p^{(t)}}{\lambda^{(t)}_{r+1}}+ \eta^{(t)}\frac{4K}{\lambda_{r+1}^{(t)}}\sqrt{\frac{\log (2p/\delta)}{|\SM^{(t)}|}}.
		\]
		
		The convergence rate can be derived directly from Theorem~\ref{thm:univ_frm}.
	\end{proof}
	
	\begin{proof} {\bf{of Theorem~\ref{thm:Reg_subnewton}}}\\
		We have
		\begin{align*}
		{}&\|(\nabla^2F(x^{(t)}) - H^{(t)}) [H^{(t)}]^{-1}\| \\
		={}& \|(\nabla^2F(x^{(t)}) -H_{|\SM^{(t)}|}+H_{|\SM^{(t)}|}- H^{(t)}) [H^{(t)}]^{-1}\| \\
		\leq{}&\|(\nabla^2F(x^{(t)}) -H_{|\SM^{(t)}|}\|\cdot\|[H^{(t)}]^{-1}\| + \|I-H_{|\SM^{(t)}|}[H^{(t)}]^{-1}\|
		\end{align*}
		
		By Lemma~\ref{lem:matrix_bnd} and $\|[H^{(t)}]^{-1}\| = 1/(\lambda_{p}^{(t)}+\alpha)$ , with probability at least $1-\delta$, it holds that
		\[
		\|(\nabla^2F(x^{(t)}) -H_{|\SM^{(t)}|}\|\cdot\|[H^{(t)}]^{-1}\| \leq \frac{4K}{\lambda_{p}^{(t)}+\alpha}\sqrt{\frac{\log (2p/\delta)}{|\SM^{(t)}|}}.
		\]
		Besises, we have
		\[
		\|I-H_{|\SM^{(t)}|}[H^{(t)}]^{-1}\| = 1-\frac{\lambda_p^{(t)}}{\lambda^{(t)}_{p}+\alpha}.
		\]
		Hence, $\gamma^{(t)}$ has the following upper bound:
		\[
		\gamma^{(t)} \leq 1-\frac{\lambda_p^{(t)}}{\lambda^{(t)}_{p}+\alpha}+ \frac{4K}{\lambda_{p}^{(t)}+\alpha}\sqrt{\frac{\log (2p/\delta)}{|\SM^{(t)}|}}.
		\]
		
		The convergence rate can be derived directly from Theorem~\ref{thm:univ_frm}.
	\end{proof}
	\begin{proof} {\bf of Theorem~\ref{thm:inexact_p}}\\
		We denote $p^* = [H(x^{(t)})]^{-1}\nabla F(x^{t})$. By Equation~\eqref{eq:res_def}, we have
		\begin{align*}
		{}&\|\nabla^2F(x^{(t)})p^{(t)} - \nabla F(x^{(t)})\| \\
		={}&\|\nabla^2F(x^{(t)})(p^* + p^{(t)}-p^*) - \nabla F(x^{(t)}) \| \\
		\leq{}&\|\nabla^2F(x^{(t)})[H(x^{(t)})]^{-1}\nabla F(x^{(t)}) - \nabla F(x^{(t)})\| + \| \nabla^2F(x^{(t)})( p^{(t)}-p^* )\|\\
		\leq{}& \epsilon_1\|\nabla F(x^{(t)})\| + \| (\nabla^2F(x^{(t)}) - H^{(t)})(p^{(t)} - p^*) \| + \|H^{(t)}(p^{(t)} - p^*)\| \\
		\leq{}&  (\epsilon_0+\epsilon_1)\|\nabla F(x^{(t)})\| + \| (\nabla^2F(x^{(t)}) - H^{(t)})[H^{(t)}]^{-1}H^{(t)}(p^{(t)} - p^*)\| \\
		\leq{}& (\epsilon_0+\epsilon_1)\|\nabla F(x^{(t)})\| + \|(\nabla^2F(x^{(t)}) - H^{(t)})[H^{(t)}]^{-1}\|\cdot\|H^{(t)}(p^{(t)} - p^*)\| \\
		\leq{}& (\epsilon_0+\epsilon_1)\|\nabla F(x^{(t)})\| + \epsilon_0\epsilon_1\|\nabla F(x^{(t)})\|\\
		={}& (\epsilon_0 +(1+\epsilon_0)\epsilon_1)\|\nabla F(x^{(t)})\|\\
		={}& \gamma^{(t)} \|\nabla F(x^{(t)})\|.
		\end{align*}
		Using Theorem~\ref{thm:univ_frm} with $\gamma^{(t)} =\epsilon_0 +(1+\epsilon_0)\epsilon_1$, we get the convergence properties. 	
	\end{proof}
	
	\begin{proof}{\bf of Theorem~\ref{thm:sub_gradient}} \\
		First, we denote $p^* = [H(x^{(t)})]^{-1}\nabla F(x^{t})$. By Equation~\eqref{eq:res_def}, we have
		\begin{align*}
		{}&\|\nabla^2F(x^{(t)})p^{(t)} - \nabla F(x^{(t)})\| \\
		={}&\|\nabla^2F(x^{(t)})(p^* + p^{(t)}-p^*) - \nabla F(x^{(t)}) \| \\
		\leq{}&\|\nabla^2F(x^{(t)})[H(x^{(t)})]^{-1}\nabla F(x^{(t)}) - \nabla F(x^{(t)})\| + \| \nabla^2F(x^{(t)})( p^{(t)}-p^* )\|\\
		\leq{}& \epsilon_1\|\nabla F(x^{(t)})\| + \| \nabla^2F(x^{(t)})[H(x^{(t)})]^{-1}(\text{g}(x^{(t)}) - \nabla F(x^{(t)}))\| \\
		\leq{}&  \epsilon_1\|\nabla F(x^{(t)})\| + \epsilon_0\|\nabla^2F(x^{(t)})[H(x^{(t)})]^{-1}\| \|\nabla F(x^{(t)})\| \\
		={}& \gamma^{(t)} \|\nabla F(x^{(t)})\|
		\end{align*}
		Using Theorem~\ref{thm:univ_frm} with $\gamma^{(t)} =\epsilon_0\|\nabla^2F(x^{(t)})[H(x^{(t)})]^{-1}\|+\epsilon_1$, we get the convergence properties. 
	\end{proof}
	\section{Convergence Analysis of ReSubNewton and ReSkeNewton}\label{sec:proof_resubnewton}
	\begin{proof} {\bf{of Theorem~\ref{thm:sub_refine}}}\\
		By Lemma~\ref{lem:matrix_bnd}, when $|\mathcal{S}| \geq \frac{16K^2\log(2p/\delta)}{\sigma^2 \epsilon^2}$,  $H^{(t)}$ in Algorithm~\ref{alg:H_subsamp_iter} has the following property:
		\[
		\|H^{(t)} - \nabla^{2}F(x^{(t)})\| \leq \epsilon \sigma.
		\]
		Above property implies the following:
		\begin{align*}
		{}&\max_{\|x\| = 1, x \in\RB^{p}} |x^{T}(H^{(t)} - \nabla^{2}F(x^{(t)}))x| \leq \epsilon\sigma,\\
		\Rightarrow{}&-\epsilon\sigma \leq \max_{\|x\| = 1, x \in\RB^{p}} x^{T}(H^{(t)} - \nabla^{2}F(x^{(t)}))x \leq \epsilon\sigma \\
		\Rightarrow{}& -\epsilon\sigma \leq  x^{T}(H^{(t)} - \nabla^{2}F(x^{(t)}))x \leq \epsilon\sigma ,\text{ for all } \|x\| = 1, x \in\RB^{p}\\
		\Rightarrow{}	& H^{(t)} - \epsilon\sigma I\preceq \nabla^{2}F(x^{(t)}) \preceq H^{(t)} + \epsilon\sigma I\\
		\Rightarrow{}&(1-\epsilon)H^{(t)}  \preceq \nabla^{2}F(x^{(t)}) \preceq (1+\epsilon) H^{(t)}
		\end{align*}
		By Lemma~\ref{lem:refine}, after $k$ iterations of inner loop of Algorithm~\ref{alg:H_subsamp_iter}, we have
		\begin{equation}
		\|x^{(k)} -x^{*}\|_A \leq \epsilon^{k} \|x^{*}\|_A. \label{eq:conv}
		\end{equation}
		
		We also have
		\begin{align}
		&\|x^{(k)} -x^{*}\|_A \geq \sqrt{\sigma}\|x^{(k)} -x^{*}\|; \label{eq:x-x*}\\
		&x^{*} = [\nabla^{2}F(x^{(t)})]^{-1}\nabla F(x^{(t)}); \notag\\
		&\|x^{*}\|_A = \|\nabla F(x^{(t)})^T [\nabla^{2}F(x^{(t)})]^{-1} \nabla F(x^{(t)})\|^{1/2} \leq 1/\sqrt{\sigma}\|\nabla F(x^{(t)})\|. \label{eq:x*}
		\end{align}
		Hence, to satisfy the condition $\|r\|\leq tol$, it needs
		\begin{align}
		\|\nabla^{2}F(x^{(t)})(x^{(k)} - x^{*})\| \leq K\|x^{(k)} - x^{*}\| \leq tol \label{eq: resi}
		\end{align}
		Combining \eqref{eq:conv}, \eqref{eq:x-x*}, \eqref{eq:x*} and \eqref{eq: resi}, we reach the result that if
		\[
		k\geq \frac{\log \frac{K\|\nabla F(x^{(t)})\|}{\sigma tol}}{\log \frac{1}{\epsilon}},
		\]
		then $\|r\|\leq tol$.
		
		The convergence properties can be derived directly from Theorem~\ref{thm:univ_frm}.
	\end{proof}
	
	\begin{proof} {\bf{of Theorem~\ref{thm:sketch_refine}}}
		If $S$ is an $\epsilon$-subspace embedding matrix for $B(x^{(t)})$, then we have
		\begin{equation}
		(1-\epsilon) \nabla^{2}F(x^{(t)})\preceq [B(x^{(t)})]^{T}S^{T}SB(x^{(t)}) \preceq (1+\epsilon) \nabla^{2}F(x^{(t)})\label{eq: preceq_1}
		\end{equation}
		By simple transformation and omitting $\epsilon^{2}$, \eqref{eq: preceq_1} can be transformed into
		\[
		(1-\epsilon) [B(x^{(t)})]^{T}S^{T}S\nabla^{2}B(x^{(t)}) \preceq \nabla^{2}F(x^{(t)}) \preceq (1+\epsilon) [B(x^{(t)})]^{T}S^{T}SB(x^{(t)})
		\]
		The rest of proof is the same to that of Theorem~\ref{thm:sub_refine}.
	\end{proof}
	
	\begin{proof}{\bf{of Corollary~\ref{cor:rand_samp}}}
		We first define $A = B^{T}B$. We sample rows of $B$ uniformly and construct random sampling sketching matrix $\S$ as in Subsection~\ref{subsec:ske_mat}. Then, we have $\tilde{A} = B^TS^TSB = \frac{n}{|\mathcal{S}|}\sum_{i\in\mathcal{S}}B(i,:)^TB(i,:)$. In the construction of $S$, we need that $\frac{1}{n} \geq \beta\ell_i$ for all $i=1,\dots,n$, where $\ell_i$ is the $i$-th leverage score of $B$, hence, $1/\beta = \mu(B)$. So, when $|\mathcal{S}| \geq \OM(\frac{\mu(B)p\log(p/\delta)}{\epsilon^2})$, we have 
		\begin{align*}
		&(1-\epsilon)\tilde{A}\preceq A\preceq(1+\epsilon)\tilde{A} \\
		&(1-\epsilon)\frac{1}{n}\tilde{A}\preceq \frac{1}{n}A\preceq(1+\epsilon)\frac{1}{n}\tilde{A}\\
		&(1-\epsilon)H^{(t)} \leq \nabla^{2}F(x^{(t)}) \leq (1+\epsilon)H^{(t)} 
		\end{align*}
		The last Equation is because $\nabla^{2}F(x^{(t)}) = \frac{1}{n}B^TB$ and $H^{(t)} = \frac{1}{|\mathcal{S}|}\sum_{i\in\mathcal{S}}B(i,:)^TB(i,:)$. The rest of proof is the same to that of Theorem~\ref{thm:sub_refine}.
	\end{proof}
	\section{Examples of Analyzing Convergence Properties of some variants of Sub-sampled Newton methods} \label{sec:app_example}
	First we use Theorem~\ref{thm:univ_frm} to analysis the local convergence properties of  Algorithm~\ref{alg:H_subsamp}.
	\begin{theorem}\label{thm:H_subsamp}
		If Eqn.~\eqref{eq:k} and Eqn.~\eqref{eq:sigma} hold and let $0<\delta<1$	and $0<\epsilon<1/2$ be given. $|\SM|$ is set as Algorithm~\ref{alg:H_subsamp} and $H^{(t)}$ is constructed as in Algorithm~\ref{alg:H_subsamp}. Then for $t = 1,\dots, T$, we have the following convergence properties:
		\begin{enumerate}
			\item If $0<\epsilon<1/2$, sequence $\{x^{(t)}:t = 1,\dots,T\}$ converges linearly with probability $(1-\delta)^{T}$.
			\item If $\epsilon \to 0$ as $t$ grows, then  sequence $\{x^{(t)}:t = 1,\dots,T\} $ converges superlinearly with probability $(1-\delta)^{T}$.
			\item If $0<\epsilon<1/2$ is a constant, and $\nabla^{2}F(x^{(t)})$ is Lipschitz continuous, then  sequence $\{x^{(t)}:t = 1,\dots,T\}$ has a linear-quadratic convergence rate  with probability $(1-\delta)^{T}$.
		\end{enumerate}
	\end{theorem}
	\begin{proof} {\bf{of Theorem~\ref{thm:H_subsamp}}}
		By Lemma~\ref{lem:matrix_bnd} with $c = \sigma$, we have
		\begin{align*}
		&\|(\nabla^2F(x^{(t)}) - H^{(t)})\| \leq \epsilon \sigma\\
		&\|[H^{(t)}]^{-1}\| \leq \frac{1}{(1-\epsilon)\sigma}
		\end{align*}
		By above result and the definition of $\gamma^{(t)}$ in Theorem~\ref{thm:univ_frm}, we obtain
		\[
		\gamma^{t} \leq \|(\nabla^2F(x^{(t)}) - H^{(t)})\|\cdot\|[H^{(t)}]^{-1}\|  \leq \frac{\epsilon}{1-\epsilon} < 1.
		\]
		The convergence rate can be derived directly from Theorem~\ref{thm:univ_frm}.
	\end{proof}
	
	For  Sketch Newton method \cite{pilanci2015newton}, a similar result can be reached.
	\begin{theorem}\label{thm:Sketch_newton}
		If Eqn.~\eqref{eq:k} and Eqn.~\eqref{eq:sigma} hold and let $0<\delta<1$ and $0<\epsilon<1/2$ be given. Assume Hessian matrix is of the form $B(x^{(t)})^TB(x^{(t)})$ and $B(x^{(t)})$ is available, where $B(x^{(t)})$ is a matrix  of dimension $n\times p$, $S\in\RB^{\ell \times n}$ is an $(\epsilon\frac{\sigma}{K})$-subspace embedding matrix for $B(x^{(t)})$ with probability at least $1-\delta$. Then for $t = 1,\dots, T$, Algorithm~\ref{alg:sketch_newton} has the following convergence properties:	
		\begin{enumerate}
			\item If $0<\epsilon<1/2$, sequence $\{x^{(t)}:t = 1,\dots,T\}$ converges linearly with probability $(1-\delta)^{T}$.
			\item If $\epsilon \to 0$ as $t$ grows, then  sequence $\{x^{(t)}:t = 1,\dots,T\}$ converges superlinearly with probability $(1-\delta)^{T}$.
			\item If $0<\epsilon<1/2$ is a constant, and $\nabla^{2}F(x^{(t)})$ is Lipschitz continuous, then  sequence $\{x^{(t)}:t = 1,\dots,T\}$ has a linear-quadratic convergence rate  with probability $(1-\delta)^{T}$.
		\end{enumerate}
	\end{theorem}
	\begin{proof}
		We give the SVD decomposition of $B(x^{(t)})$ as follow:
		\[
		B(x^{(t)}) = U\Sigma V^T.
		\]
		We give the bound of $\|\nabla^2F(x^{(t)}) - H^{(t)}\|$ as follow:
		\begin{align*}
		\|\nabla^2F(x^{(t)}) - H^{(t)}\| = \|\Sigma(I - U^TS^TSU)\Sigma\| \leq K*\frac{\sigma}{K}\epsilon = \epsilon\sigma,
		\end{align*}
		where the inequality follows from the property of  $(\epsilon\frac{\sigma}{K})$-subspace embedding.
		
		For $\|[H^{(t)}]^{-1}\|$, we have
		\[
		\|[H^{(t)}]^{-1}\| \leq \frac{1}{(1-\epsilon)\sigma}
		\]
		Hence, $\gamma^{(t)}$ can be bounded as follow:
		\[
		\gamma^{(t)} \leq \|(\nabla^2F(x^{(t)}) - H^{(t)})\|\cdot\|[H^{(t)}]^{-1}\|  \leq \epsilon\sigma\cdot\frac{1}{(1-\epsilon)\sigma} =  \frac{\epsilon}{1-\epsilon} < 1
		\]
		The convergence rate can be derived directly from Theorem~\ref{thm:univ_frm}.
	\end{proof}
	
	\bibliographystyle{plainnat}
	\bibliography{referee}
\end{document}